\documentclass[12pt,english]{amsart}
\usepackage[utf8]{inputenc}
\usepackage{amsmath,amssymb,amsthm}
\usepackage{xcolor}
\usepackage{amsmath}
\usepackage{hyperref}
\makeatletter
\tagsleft@false % Force tags to the right
\makeatother
 \usepackage{comment}
\newtheorem{lemma}{Lemma}[section]
\newtheorem{prop}{Proposition}[section]
\newtheorem{exmp}{Example}[section]
\newtheorem{thm}{Theorem}[section]
\newtheorem{cor}{Corollary}[section]
\newtheorem{dfn}{Definition}[section]
\newtheorem{rmk}{Remark}[section]
\newcommand{\shP}{\mathcal{P}}
\newcommand{\shL}{\mathcal{L}}
\newcommand{\CC}{\mathbb{C}}

\newcommand{\R}{{\mathbb R}}

\newcommand{\be}{\begin{equation}}
\newcommand{\ee}{\end{equation}}

\newcommand{\ba}{\begin{eqnarray}}
\newcommand{\ea}{\end{eqnarray}}

\begin{document}

\title[Unitarity in quantization for toric manifolds]{A new look at unitarity in quantization commutes with reduction for toric manifolds}

{}

\author[Mour\~ao]{José M. Mour\~ao}
\email{jmourao@tecnico.ulisboa.pt}
\address{Department of Mathematics and Center for Mathematical Analysis, Geometry and Dynamical Systems, Instituto Superior Técnico, Universidade de Lisboa, 1049-001 Lisboa, Portugal} 

\author[Nunes]{João P. Nunes}
\email{jpnunes@math.tecnico.ulisboa.pt}
\address{Department of Mathematics and Center for Mathematical Analysis, Geometry and Dynamical Systems, Instituto Superior Técnico, Universidade de Lisboa, 1049-001 Lisboa, Portugal} 

\author[Pereira]{Augusto Pereira}
\email{augusto.pereira@iseg.ulisboa.pt}
\address{Department of Mathematics, Instituto Superior de Economia e Gest\~ao, 
Universidade de Lisboa, 1200-781 Lisboa, Portugal}

\author[Wang]{Dan Wang}
\email{dan.wang@tecnico.ulisboa.pt;dwang116@link.cuhk.edu.hk}
\address{Department of Mathematics and Center for Mathematical Analysis, Geometry and Dynamical Systems, Instituto Superior Técnico, Universidade de Lisboa, 1049-001 Lisboa, Portugal;
Department of Mathematics, The Chinese University of Hong
Kong, Shatin, Hong Kong} 
%\date{\today}
\maketitle
\begin{abstract}
For a symplectic manifold $M$ with the Hamiltonian action of a compact Lie group $G$,
we give a new perspective on the problem of unitarity in the quantization commutes with reduction correspondence. 
We argue that this correspondence leads naturally to the introduction of what we call  Fourier polarizations,  whose real directions 
include  the Hamiltonian vector fields of all invariant functions of the moment map. 

For a Fourier polarization, quantization commutes with reduction  unitarily.  
For a general $G$-invariant polarization,  $\mathcal{P}$,
unitarity in the quantization commutes with reduction correspondence can then be formulated in terms of the existence of a Fourier polarization, $\mathcal{P}_F$, such that 
the quantizations with respect to $\mathcal{P}$ and to $\mathcal{P}_F$ are unitarily equivalent.

We describe this perspective in detail for  symplectic toric manifolds, by considering symplectic reductions associated to 
the action of a subtorus $T^p\subset T^n$. We consider half-form quantization in Fourier (mixed) polarizations $\mathcal{P}_\infty$, whose real directions are generated by (the Hamiltonian vector fields of the) components of the $T^p$ moment map.
We show that Fourier mixed polarizations of this type can be obtained 
by  Mabuchi geodesic rays of toric K\"ahler structures, starting at a given initial structure and with velocity given by the norm square of the moment map of the torus subgroup. 
We lift these geodesic rays to the quantum bundle via generalized coherent state transforms, which define equivariant isomorphisms between quantum Hilbert spaces for the  K\"ahler  and   Fourier polarizations, where the later are obtained at infinite geodesic time. 
Unitarity in quantization commutes with reduction 
in these cases 
then becomes equivalent to the unitarity
of these isomorphisms. In general, however, in the toric case, these
isomorphisms are not unitary.

\end{abstract}

{}
\tableofcontents

\section{Introduction and main results}

{}

An important object of study in geometric quantization is the correspondence commonly known as ``quantization commutes with reduction": given a symplectic manifold $(M,\omega)$ with an Hamiltonian action of a compact Lie group $G$ and equivariant moment map $$\mu:M\to (\mathrm{Lie}\, G)^*,$$ one wants to find an identification between the quantization of the Marsden-Weinstein symplectic reduction, 
$$
M_0 := \mu^{-1} (0) / G,
$$
and the subspace of the quantization of $(M,\omega)$  which is invariant under $G$.
Recall that, in the general framework of geometric quantization, as briefly reviewed below, one considers a complex line bundle $L\to M$, the so-called prequantum bundle, equipped with an hermitean structure and compatible connection $\nabla$ of curvature $-i\omega$. 
When one refers to a quantization of $(M,\omega)$ one has implicitly chosen a $G$-invariant polarization, that is an involutive $G$-invariant Lagrangian distribution $\mathcal{P}\subset TM\otimes \mathbb{C}$, so that the quantization of $(M,\omega)$ with respect to $\mathcal{P}$ is given by the space sections of $L$ which are covariantly constant along $\mathcal{P},$ $\mathcal{H}_\mathcal{P}$. The choice of a $G$-invariant $\mathcal{P}$ is, therefore, implicit when one formulates the problem of ``quantization commutes with reduction". Moreover, one assumes that $\mathcal{P}$ induces a polarization on the reduction $M_0.$ In the important case when $\mathcal{P}$ is a $G$-invariant K\"ahler polarization on $M$, $M_0$ gets an induced K\"ahler polarization, along with the remaining data of a prequantum line bundle equipped with hermitean connection \cite{GS1}, and it is with respect to these induced data that one quantizes the (zero level) symplectic reduction. 

A very important aspect of ``quantization commutes with reduction"  is then the one of unitarity in the correspondence between the quantization of $M_0$, $\mathcal{H}_{M_0}$, and the appropriate subspace, $\mathcal{H}_\mathcal{P}^0$, of the quantization of $M.$ (For a study of the problem of unitarity in the quantization commutes with reduction correspondence, with an emphasis on the semi-classical limit, see \cite{HK1}.) 

The quantization commutes with reduction correspondence can also be studied for symplectic reductions at other levels of the moment map. Let  
$$
M_\lambda = \mu^{-1}(\mathcal{O}_\lambda)/G ,\, \lambda\in (\mathrm{Lie} G)^*,
$$
where $\mathcal{O}_\lambda \subset (\mathrm{Lie} \, G)^*$ is the coadjoint orbit through $\lambda\in (\mathrm{Lie}\, G)^*$. A $G$-invariant K\"ahler structure on $M$ also induces a natural K\"ahler structure on $M_\lambda$  by means of the ``shifting trick", where $M_\lambda$ is identified with the symplectic reduction at zero level of the symplectic manifold $M\times \mathcal{O}_\lambda$, and one uses the negative of the Kostant-Kirillov symplectic structure on $\mathcal{O}_\lambda$. (See \cite{GS1} and, for example, \cite{MM}.) For integral $\lambda$, the quantization of $M_\lambda$ then measures the multiplicity with which the irreducible representation of $G$ with highest weight $\lambda$ appears in the isotypical decomposition of the quantization of $M$ with respect to $\mathcal{P}$,
$$
\mathcal{H}_\mathcal{P} = \bigoplus_{\lambda\in \hat G} \mathcal{H}_\mathcal{P}^\lambda,
$$
where $\hat G$ is the set of equivalence classes of irreducible representations of $G$. (See Theorem 6.2 in \cite{GS1}.) For instance, if $M$ is multiplicity free (see \cite{GS2}) then, if not empty, $M_\lambda$ is a point, indicating that the irreducible representation $\lambda\in \hat G$ occurs with multiplicity one in $\mathcal{H}_P.$ 
{}

In the present paper, we propose a new look at the unitarity problem in  quantization commutes with reduction, by using a certain type of polarizations which we call ``Fourier polarizations". 
Consider a symplectic manifold $(M,\omega)$ with an Hamiltonian action of a compact Lie group $G$ and equivariant moment map $\mu$. 
One has the property that on an open dense subset of ``regular points", in a suitable sense, $M_\mathrm{reg}\subset M$, there is a smooth Hamiltonian action of a torus, $T_\mathrm{inv}$, with moment map $\mu_\mathrm{inv}$ such that, for $p\in M_\mathrm{reg}$,
$$
\mu_\mathrm{inv} (p) = \lambda 
$$
where $\lambda$ is the unique element in the positive Weyl chamber in $(\mathrm{Lie} \,G)^*$ such that $\mu(p)\in \mathcal{O}_\lambda.$
(See, for example, \cite{L}, for a study of properties of the Hamiltonian action of $T_\mathrm{inv}$ and for further references.)
Note that, for $ \mu^{-1} (\mathcal{O}_\lambda) \subset M_\mathrm{reg}$, 
$\mu^{-1}(\mathcal{O}_\lambda) = \mu_\mathrm{inv}^{-1}(\lambda)$ and 
the reductions with respect to the $T_\mathrm{inv}$ action can be described by (see Section 26, and in particular Theorem 26.6, in \cite{GS4})
$$
\mu_\mathrm{inv}^{-1} (\lambda) / T_\mathrm{inv} \cong M_\lambda \times \mathcal{O}_\lambda  .
$$
We will use the term ``Casimir functions" for functions on $M_\mathrm{reg}$ of the form
$$u\circ \mu = u\circ \mu_\mathrm{inv}: M_\mathrm{reg}\to \mathbb{R},$$ where $u: (\mathrm{Lie}\, G)^*\to \mathbb{R}$ is $\mathrm{Ad}^*-$invariant.   (The simplest example would be the function $\vert\vert \mu\vert\vert^2:M\to \mathbb{R}$, for the norm on $(\mathrm{Lie}\, G)^*$ induced by the Killing form.) Note that Casimir functions are constant along the level sets $ \mu^{-1} (\mathcal{O}_\lambda).$

\begin{dfn}\label{defFourierpolariz}
We say that a $G$-invariant polarization on $(M_\mathrm{reg},\omega)$ is a Fourier polarization if it contains in its real directions the space spanned by the Hamiltonian vector fields of all the Casimir functions.\footnote{In \cite{BHKMN}, the term ``Fourier polarization" was used in a more strict sense, see Definition 3.18 in that reference. There, one also demands that the remaining transverse directions of the Fourier polarization $\mathcal{P}_F$ are given by the K\"ahler polarizations induced on the reductions $M_\lambda\times \mathcal{O}_\lambda$ by an invariant K\"ahler polarization on $M$, as described above. In the toric case that we detail in this paper, where coadjoint orbits are trivial, the Fourier polarization does satisfy this more strict Definition 3.18 of \cite{BHKMN}.}  We will moreover assume that a  Fourier  polarization induces ($G$-invariant) polarizations on the coisotropic reductions $\mu_\mathrm{inv}^{-1} (\lambda)/T_\mathrm{inv}\cong M_\lambda \times \mathcal{O}_\lambda$.
\end{dfn}

{}

\begin{rmk}
 The assumption that the Fourier polarization induces polarizations on the coisotropic reductions holds not only for the mixed toric polarizations considered in this paper but also for the mixed polarizations considered in \cite{BHKMN, BFHMN}.   
\end{rmk}

Fourier polarizations are the most natural ones to be considered when one is interested in discussing quantization and symplectic reduction. The classical observables given by Casimir functions have constant values along the level sets $\mu^{-1}(\mathcal{O}_\lambda)$. It is therefore natural to consider quantizations of $M_\lambda$ where the quantum operators corresponding to the Casimir functions correspond just to (constant) multiplication operators. This is precisely what happens for Fourier polarizations. 

Let $\mathcal{P}_F$ be a Fourier polarization on $M_\mathrm{reg}$. By construction, the Kostant-Souriau prequantum operators for Casimir functions (see Section \ref{section:geometric-quantization}), act just as multiplication operators on the space of $\mathcal{P}_F$-polarized sections of the prequantum bundle $L\to M$. Indeed, it is immediate to obtain that the first-order differential operator part of these prequantum operators annihilates $\mathcal{P}_F$-polarized sections, so that only the zeroth order 
multiplication operator survives. Therefore,  the Hilbert space for the quantization of $M$ with respect to $\mathcal{P}_F$ will have a spectral decomposition 
into a direct sum (or a direct integral) 
\begin{equation}\label{spectral}
\mathcal{H}_{\mathcal{P}_F} = \bigoplus_{\lambda\in \mathcal{S}} \mathcal{H}_{\mathcal{P}_F}^\lambda,    
\end{equation}
where $\mathcal{S}$ labels the Bohr-Sommerfeld set of quantizable coisotropic reductions, $\left\{\mu_\mathrm{inv}^{-1}(\lambda)/T_\mathrm{inv}\right\}_{\lambda\in \mathcal{S}}$, and where a Casimir function $u\circ \mu$ acts on the component $\mathcal{H}_{\mathcal{P}_F}^\lambda$ just by multiplication by the constant $u(\lambda).$\footnote{In the case of symplectic torus manifolds, that is detailed in this paper, this corresponds to the decomposition into weight spaces for the action of the subtorus, $T^p\subset T^n$, with respect to which one is taking symplectic reductions.} Note that these values of the Casimir functions imply that the component 
$\mathcal{H}_{\mathcal{P}_F}^\lambda$ corresponds to the irreducible representation $\lambda\in \hat G$ in the isotypical decomposition for $\mathcal{H}_{\mathcal{P}_F}.$
Note also, accordingly, that $\mathcal{P}_F$-polarized sections of $L$ will have support along the level sets of $\mu_\mathrm{inv}$, $\mu_\mathrm{inv}^{-1}(\lambda)$, which satisfy Bohr-Sommerfeld conditions of trivial holonomy for the prequantum connection. 
In addition, they will be polarized with respect to the polarization induced by $\mathcal{P}_F$ on the coisotropic reductions $\mu_\mathrm{inv}^{-1}(\lambda)/T_\mathrm{inv}.$
The component $\mathcal{H}_{\mathcal{P}_F}^\lambda$ then corresponds to the quantization of $\mu_\mathrm{inv}^{-1}(\lambda)/T_\mathrm{inv}\cong M_\lambda\times \mathcal{O}_\lambda$. 

In this sense, quantization with respect to $\mathcal{P}_F$ is quantization after reduction. Moreover, given hermitian structures on each component $\mathcal{H}_{\mathcal{P}_F}^\lambda$, one can define naturally an hermitean structure on 
$\mathcal{H}_{\mathcal{P}_F}$ such that the linear isomorphism in (\ref{spectral}) is unitary. (For instance, in the case when the polarizations induced by $\mathcal{P}_F$ on $M_\lambda$ are K\"ahler, one will have a natural hermitian structure on $\mathcal{H}_{\mathcal{P}_F}^\lambda$ given by standard K\"ahler quantization of 
$M_\lambda\times \mathcal{O}_\lambda$.)

The problem of unitarity in the quantization commutes with reduction correspondence, with respect to a given $G$-invariant polarization $\mathcal{P}$ on  $(M,\omega)$, can then be formulated as the question:

\bigskip
``Is there a Fourier polarization on $(M,\omega)$, $\mathcal{P}_F$,  so that quantizations with respect to $\mathcal{P}$ and to $\mathcal{P}_F$ are unitarily equivalent ?"
\bigskip

In particular, this corresponds to the existence of unitary isomorphisms 
$$
\mathcal{H}_{\mathcal{P}}^\lambda \cong \mathcal{H}_{\mathcal{P}_F}^\lambda, 
$$
for all $\lambda\in \hat G.$  Note that in this approach we consider all symplectic reductions at once. 
{}

More concretely, when $M$ has a $G$-invariant K\"ahler polarization, $\mathcal{P}$, let $\mathcal{P}_{M_\lambda}$ be the K\"ahler polarization induced on  the reduction $M_\lambda$, as mentioned above. 
With respect to the polarization $\mathcal{P}$, the statement that the quantization commutes with reduction  correspondence is unitary, can then be addressed in the following steps:
\begin{itemize}
    \item[$i)$] Find a Fourier polarization on $M_\mathrm{reg}$, $\mathcal{P}_F$, whose coisotropic reductions give 
    on each $M_\lambda\times \mathcal{O}_\lambda, \lambda\in \mathcal{S},$ the polarizations 
    $\mathcal{P}_{M_\lambda} \times \mathcal{P}_{\mathcal{O}_\lambda}$, where $\mathcal{P}_{\mathcal{O}_\lambda}$ is the standard K\"ahler polarization on $\mathcal{O}_\lambda$.
     Then, we obtain isomorphisms 
     \begin{equation}\label{isom1}\mathcal{H}_{M_\lambda}\otimes V_\lambda\cong \mathcal{H}_{\mathcal{P}_F}^\lambda,\end{equation}
     where $V_\lambda$ is the irreducible representation with highest weight $\lambda$ which, by the Borel-Weil construction, correponds to the quantization of $\mathcal{O}_\lambda$ with respect to $\mathcal{P}_{\mathcal{O}_\lambda}$.

     As described above, one can use these isomorphisms to define the hermitean structure on $\mathcal{H}_{\mathcal{P}_F}$, so that for $\mathcal{P}_F$ itself the quantization commutes with reduction correspondence will become automatically unitary. In fact, as we will show below, in the toric case (where co-adjoint orbits are trivial), 
     $\mathcal{P}_F=\mathcal{P}_\infty$ is obtained as a limit point in a family of toric K\"ahler polarizations and it inherits a natural hermitean structure for which, indeed, the isomorphisms (\ref{isom1}) are unitary. (see Theorem \ref{thmqcr} and Corollary \ref{cor_qrunitary}.)

    \item[$ii)$] Check if the quantizations of $M$ with respect to $\mathcal{P}$ and $\mathcal{P}_F$ are unitarily equivalent, which is equivalent to checking if the  
    isomorphisms $$\mathcal{H}_{M_\lambda}\otimes V_\lambda \cong \mathcal{H}_{\mathcal{P}_F}^\lambda \cong \mathcal{H}_{{\mathcal P}}^\lambda$$ are unitary.
    (In the toric case the answer, in general, is negative.)
\end{itemize}

{}

In Section 7 of \cite{BHKMN} it is conjectured that, from an initial $G$-invariant K\"ahler polarization, one can reach a Fourier polarization $\mathcal{P}_F$ by taking Hamiltonian flows in imaginary time generated by an Hamiltonian function which is a strictly convex Casimir function. In \cite{BHKMN}, this was shown to be true in the case when $M=T^*K$ is the cotangent bundle of a compact Lie group while in \cite{BFHMN} it is shown that this fact also holds for cotangent bundles of symmetric spaces of compact type. 
(For $T^*K$ with standard K\"ahler structure the answer to $ii)$ is affirmative, see \cite{BHKMN}.)
This conjecture for a more general situation is currently under further investigation.

In this paper, we will address the issue of unitarity in quantization commutes with reduction for symplectic toric manifolds from the perspective that has been expounded above. As we will show, Fourier polarizations associated to the action of a subtorus can indeed be described by Hamiltonian flows in imaginary time. 

%Interesting relations between Hamiltonian flows in imaginary time and geometric quantization on K\"ahler manifolds have been explored in %recent years (see e.g. \cite{KMN1,KMN4,LW1,LW2, LW3,P,W1}). 
These ``flows" describe geodesics for the Mabuchi metric on the space of K\"ahler metrics on $M$ which often converge, at infinite geodesic time, to interesting real or mixed polarizations. 
Besides leading to interesting polarizations, the Mabuchi geodesics
have frequently the 
very important property, by the use of generalized coherent state transforms, of allowing
natural lifts to the quantum bundle, that is the bundle of Hilbert spaces of quantum states for the geodesic family of polarizations $\mathcal{P}_t, t\geq 0,$ 
with fiber $\mathcal{H}_{\mathcal{P}_t}$ over $t$. These transforms can then be used to compare quantizations for different polarizations.

A particularly nice class of K\"ahler manifolds to study with these methods is the class of toric K\"ahler manifolds $(M, \omega,J)$, with Hamiltonian $T^{n}$ action and the corresponding moment map $\mu:M \twoheadrightarrow P$, where $P$ is the moment polytope. Here, the Mabuchi rays generated by a convex function of the moment map and starting at any initial toric K\"ahler polarization give, at infinite geodesic time, the toric real polarization.

By lifting the geodesics to the quantum bundle we get coherent state transforms
which  relate the spaces of holomorphic sections along the geodesic family and show the convergence of the monomial sections to the distributional sections of the prequantum line bundle which generate the Hilbert space of quantum states for the real toric polarization. 

This convergence has been established for $L^1$ normalized sections in \cite{BFMN} and in \cite{KMN1} it was shown for $L^2$ normalized sections if one includes the half-form correction. In \cite{KMN4}, the convergence of half-form corrected holomorphic sections is described in terms of a generalized coherent state transform. The action of a subtorus is considered in \cite{P,W1}. 
In \cite{LW1,LW2,LW3}, the study of Mabuchi geodesic families of K\"ahler structures was generalized from the toric case to the case of manifolds with an Hamiltonian torus action. There are many previous works on closely related problems (see e.g. \cite{An2,BFHMN,BHKMN,CLL,FMMN1,GS1,GS3,Hal,HK2,JW,LY1,MNP,MNR}).

In the present paper, Section \ref{sectionPinfty} describes quantization with respect to the toric mixed polarization 
$\mathcal{P}_\infty$ attained at infinite geodesic time along Mabuchi rays which, starting at an initial toric K\"ahler structure defined by a symplectic potential $g$, are obtained by Hamiltonian flow in imaginary time generated by functions $H$ on the moment polytope that are convex only along $\mathrm{Image} (\mu_p)$, where $\mu_{p}: M\twoheadrightarrow \Delta$ denotes the moment map of a subtorus $T^{p}\subset T^n$ action.
We first provide a local description of $\mathcal{P}_{\infty}$. Secondly, we identify a basis for half-form corrected quantum space $\mathcal{H}_{\mathcal{P}_\infty}$ associated  $\mathcal{P}_{\infty}$, that is labelled by the integer points of the moment polytope $P$. Thirdly, using the generalized coherent state transforms (gCST), we construct an isomorphism between the quantum spaces for the half-form corrected K\"ahler polarization and for the mixed polarization. 

Section \ref{chapter:mixed-polarizations} discusses the quantization commutes with reduction correspondence for the action of the subtorus $T^p$.
The polarization $\mathcal{P}_\infty$ is a Fourier polarization as in Definition \ref{defFourierpolariz}.
As described above, the quantization commutes with reduction correspondence is a very natural property of the quantization in the mixed polarization $\mathcal{P}_\infty$, since its coisotropic reductions occur at fixed values of the components of $\mu_p$ and since the Casimir functions, being $\mathcal{P}_\infty$-polarized, act simply by multiplication operators on 
$\mathcal{H}_{\mathcal{P}_\infty}.$ 
The Fourier polarization  $\mathcal{P}_\infty$, has real directions corresponding to the components of $\mu_p$ and complex directions corresponding to the K\"ahler reductions of 
$\mathcal{P}_0$ (thus it satisfies the property of Fourier polarizations as in Definition 3.18 in \cite{BHKMN}). As described in Section \ref{subsecqcr}, there is a natural hermitean structure on $\mathcal{H}_{\mathcal{P}_\infty}$, induced by the family of gCST,  such that, as expected, the quantization commutes with reduction correspondence for the Fourier polarization $\mathcal{P}_\infty$ is unitary. That is,  there is a unitary equivalence between the quantization of each symplectic reduction (with respect to the K\"ahler structure inherited from $\mathcal{P}_0$) and the appropriate subspace of the quantization of $M$ with respect to the Fourier polarization $\mathcal{P}_\infty$. As we show in Theorem \ref{thmqcr} and Corollary \ref{cor_qrunitary}, this unitary equivalence is obtained globally, for all (quantizable) levels of $\mu_p$ at once.

For illustration, the example $M=\mathbb{CP}_1\times \mathbb{CP}_1$, with $T^p$ being one of the $S^1$ factors, is discussed in Section \ref{p1p1example}.

Thus, as in points $i)$ and $ii)$ above, for the initial toric K\"ahler polarization $\mathcal{P}_0$ and relative to the $T^p$-action, the question of unitarity in the quantization commutes with reduction correspondence is more naturally phrased as the question of unitary equivalence between quantizations with respect to  $\mathcal{P}_0$ and to $\mathcal{P}_\infty$. As described above, and also in Section \ref{newsubsec}, we believe that this perspective should be taken more generally, whenever a K\"ahler polarization is related to a Fourier polarization. 

{}

 %Let $(M,\omega, J)$ be a toric variety determined by the Delzant polytope $P$. %Let $\mu_{p}:M \twoheadrightarrow \Delta$ be the moment map associated to the Hamiltonian subtorus $T^{p}$-action.
  %and $H=\frac{1}{2}\sum_{j=1}^{p}x_{j}^{2}$. 
Let \((M, \omega, J)\) be a toric variety determined by the Delzant polytope \(P\). Let \(\mu_p : M \twoheadrightarrow \Delta\) denote the moment map for the action of a subtorus \(T^p \subset T^n\). Let \(H\) be a smooth function on \(P\) which is convex along \(\mathrm{Image}(\mu_p)\). Consider the one-parameter family of symplectic potentials given by \(g_s = g + sH\). (Alternatively, if \(H\) is a strictly convex function on \(\mathrm{Image}(\mu_p)\), then \(H \circ \mu_p\) defines a function on \(M\) and \(g_s = g + sH \circ i_p^*\) gives the desired family of symplectic potentials. For simplicity, we will denote this more generally by \(g_s = g + sH\).)
  
  Our main results are as follows: 

\begin{thm}[Theorem \ref{prop:evol-polarization}] 
       For a choice of symplectic potential on $M$, $g$,    
     let $\mathcal{P}_s, s\geq 0$, be the family of K\"ahler polarizations associated to the symplectic potential $g_s= g+sH, s\geq 0$, which is obtained under the imaginary time flow of the Hamiltonian vector field $X_H$. Then the limiting polarization $\mathcal{P}_\infty$ exists, and
     $$\mathcal{P}_\infty:=  \lim_{s\to +\infty} \mathcal{P}_s $$ is a (singular) mixed polarization.
     Moreover, over the open dense $(\mathbb{C}^*)^n$-orbit $\mathring M,$ 
     $$
     \mathcal{P}_\infty = \langle  \partial/\partial\tilde \theta_{1}, \dots, \partial/\partial\tilde \theta_p, 
     X_{\bar{\tilde z}_{p+1}}, \dots, X_{\bar{\tilde z}_{n}},\rangle_{\mathbb C}= 
     $$
     $$
     =
     \langle  \partial/\partial\tilde \theta_{1}, \dots, \partial/\partial\tilde \theta_p, 
     \frac{\partial}{\partial{\tilde z_{p+1}}}, \dots, \frac{\partial}{\partial{\tilde z_{n}}}\rangle_{\mathbb C},
     $$
     where $\tilde \theta_j, j=1, \dots, p$ are the angle coordinates for the $T^p$ action and the $\tilde z_j, j=p+1, \dots n$ are holomorphic coordinates along the coisotropic reductions.
\end{thm}

So the polarization $\mathcal{P}_{\infty}$ is no longer real but mixed, and the holomorphic sections converge to distributional sections, which have a holomorphic factor along the complex directions of the limit polarization. In particular, when $H$ is a strictly convex function on $P$, this result coincides with \cite[Theorem 1.2]{BFMN}.

\begin{dfn}[Definition \ref{correctedQS}]
The quantum Hilbert space for the half-form corrected mixed polarization $\mathcal{P}_{\infty}$ is defined by 
$$
\mathcal{H}_{\mathcal{P}_{\infty}} =B_{\mathcal{P}_{\infty}} \otimes \sqrt{|dX_{1}^{p}\wedge dZ_{p+1}^{n}|},
$$
where $$B_{\mathcal{P}_{\infty}}=\{\sigma \in \Gamma(M, L^{-1})' \mid \nabla^{\infty}_{\bar \xi} \sigma=0, \forall \xi \in \Gamma(M, \mathcal{P}_{\infty})\}.$$
\end{dfn}

The $\mathcal{P}_\infty$-polarized sections, $\tilde \sigma^m_\infty, m\in P\cap \mathbb{Z}^n$, in (\ref{limitsections}), will be distributional along the momentum coordinates $x_1, \dots, x_p$ and they will be holomorphic along the transverse directions corresponding to the the K\"ahler reductions of $\mathcal{P}_0.$

\begin{thm}[Theorem\ref{thm_polsectionsinside}]
The distributional sections $\tilde{\sigma}^m_\infty, m\in P\cap \mathbb{Z}^n$  are in $\mathcal{H}_{\mathcal{P}_\infty}$. Moreover, 
for any $\sigma \in \mathcal{H}_{\mathcal{P}_\infty}$, $\sigma$ is a linear combination of the sections $\tilde{\sigma}^m_\infty, m\in P\cap \mathbb{Z}^n$. Therefore, the distributional sections $\{\tilde{\sigma}^m_\infty, m\in P\cap \mathbb{Z}^n\}_{m\in P\cap \mathbb{Z}^n}$ form a basis of $\mathcal{H}_{\mathcal{P}_\infty}$.
\end{thm}
This implies that the dimension of the quantum Hilbert space $\mathcal{H}_{\mathcal{P}_\infty}$ for the half-form corrected mixed polarization coincides with those of quantum spaces for real and K\"ahler polarizations.
When $p=n$, this result coincides with \cite[Theorem 4.7]{KMN1}.
Let $H^{prQ}$ be the Kostant-Souriau prequantum operator associated to $H$.
The gCST is then defined by the  $T^n-$equivariant linear isomorphism 
\begin{align*}
    U_{s} = (e^{s\hat{H}^{prQ}} \otimes e^{is \shL_{X_H}}) \circ e^{-s\hat H^Q} : \mathcal{H}_{\mathcal{P}_{0}} \to \mathcal{H}_{\mathcal{P}_{s}}, s \geq 0.
\end{align*}

\begin{thm}[Theorem \ref{convsections}]
    Let $s\geq 0.$
  \begin{align*}
      \lim_{s\to +\infty} U_{s}\sigma^m_{0} = \tilde{\sigma}^m_\infty, \, m\in P\cap \mathbb{Z}^n.
  \end{align*}
  \end{thm}
 This result establishes that there is an isomorphism $U_{\infty}$ between quantum spaces for the half-form corrected K\"ahler polarization and mixed polarization.   When $p=n$, this result coincides with \cite[Theorem 4.3]{KMN4}. 

 A natural hermitian structure can be defined on $\mathcal{H}_{{\mathcal P}_\infty}$ from the following
  
  \begin{thm}[Theorem \ref{lemma-norms}]
Let $\{\tilde{\sigma}^{m}_{s}:= U_{s}(\sigma_{0}^{m})\}_{m \in P \cap \mathbb{Z}^{n}}$ be the basis of $\mathcal{H}_{\mathcal{P}_{s}}$. Then we have:
$$\lim_{s \to \infty}\vert\vert \tilde{\sigma}^m_s\vert\vert_{L^2}^2 =  c_m \pi^{p/2},$$
where the constant $c_m$ is given by
$$
c_m =  \int_{P}  \left(\Pi_{j=1}^p\delta(x_j-m_j)\right) e^{-2((x-m)\cdot y -g)} (\det D)^\frac12dx_1\cdots dx_n.
$$
\end{thm}

{}

Along the Mabuchi rays of toric K\"ahler structures that we are considering, for any finite value of $s\geq 0$, the correspondence given by quantization commutes with reduction is not unitary. 
However, this correspondence is unitary at infinite geodesics time $s\to \infty$. Let 
$M_{\underline m}$ be the K\"ahler reduction corresponding to the level set $\mu_p^{-1}(\underline m)$, for $\underline m \in \Delta \cap \mathbb{Z}^p$ and let $\mathcal{H}_{M_{\underline m}}$ be the Hilbert space for its K\"ahler quantization, so that 
$$
\mathcal{H}_{M_{\underline m}}=\left\{ \sigma^{\underline m, m'}, (\underline m, m')\in P\cap \mathbb{Z}^n\right\}.
$$
Then,

\begin{thm}[Theorem \ref{thmqcr}]
    The natural $T^n$-equivariant  linear isomorphism
\begin{eqnarray}\nonumber
\mathcal{H}_{\mathcal{P}_\infty} &\to& \bigoplus_{\underline m\in \Delta\cap\mathbb{Z}^n}\mathcal{H}_{M_{\underline m}}\\ \nonumber
\tilde \sigma^m_\infty & \mapsto & \sigma^{\underline m, m'},
\end{eqnarray}
for $(m_1, \dots, m_p)=\underline m$ and $m=(\underline m, m')\in P\cap \mathbb{Z}^n$
    is unitary up to the overall constant $\pi^{p/2}$.
\end{thm}

As expected from the general discussion in Section \ref{newsubsec}, we thus obtain
\begin{cor}[Corollary \ref{cor_qrunitary}]
The quantization commutes with reduction correspondence for the toric mixed polarization $\mathcal{P}_\infty$ is unitary (up to an overall constant). \end{cor}

Note that, however, that the quantizations with respect to $\mathcal{P}_0$ and $\mathcal{P}_\infty$ are not unitarily equivalent, which is a restatement of the fact that the quantization commutes with reduction correspondence for the initial K\"ahler polarization $\mathcal{P}_0$ is not unitary.

We conjecture that the analogs of Theorem \ref{thmqcr} and Corollary \ref{cor_qrunitary} will hold in the more general family of examples of K\"ahler manifolds with Hamiltonian action of a torus, where the quantization along Mabuchi geodesic rays of invariant polarizations converging to a mixed Fourier polarization at infinite geodesic time were studied in \cite{LW1,LW2,LW3}. 

\section{Preliminaries}\label{section-prelim}

\subsection{Geometric quantization}\label{section:geometric-quantization}

In classical mechanics one works with real-valued functions on symplectic manifolds, respectively the physical observables on phase space, in the terminology of physicists. Geometric quantization is concerned with devising a consistent method of going from a classical system to a quantum system. The physical observables in quantum mechanics are no longer functions on phase space, but rather operators on a Hilbert space of quantum states.

\begin{dfn}
A symplectic manifold $(M,\omega)$ is \emph{quantizable} if there exists a hermitian line bundle $(L,h) \to M$ with compatible connection $\nabla$ of curvature $F_\nabla = -i\omega$. The triple $(L,\nabla,h)$ is called \emph{prequantization data}.
\end{dfn}
One can show \cite{Wo} that, in the compact case, $M$ is quantizable whenever
\begin{align}\nonumber
    \left[\frac{\omega}{2\pi}\right] \in H^2(M,\mathbb{Z}).
\end{align}
This is often referred to as the \emph{integrality condition}.
The \emph{prequantum Hilbert space} $L^2(M,L)$ is the $L^2$-completion of the space of sections of the hermitian line bundle $(L,h) \to M$ with respect to the scalar product given by
\begin{align}\nonumber
    \langle \psi, \psi \rangle : = \int_M h(\psi,\psi)\Omega,
\end{align}
where $\Omega = (1/n!)\omega^n$ is the Liouville volume form.

\begin{dfn}\label{dfn:kostant-souriau-operator}
Let $f \in C^\infty(M,\mathbb{C})$. The \emph{Kostant-Souriau prequantum operator} $\hat f^{prQ}$ is the (unbounded) operator on the prequantum Hilbert space given by
\begin{align}\nonumber
    \hat f^{prQ} = f + i\nabla_{X_f},
\end{align}
where $f$ stands for multiplication by $f$ and $X_f$ is the Hamiltonian vector field associated to $f$.
\end{dfn}

The prequantum Hilbert space, however, is too big, in the sense that the asignment $f \mapsto \hat f^{prQ}$ does not give an irreducible representation of the Poisson algebra $C^\infty(M)$. To improve on this, one needs to choose a \emph{polarization} $\mathcal{P}$ of $(M,\omega)$ that is, an integrable Lagrangian distribution in $TM \otimes \mathbb{C}$.

There are two special types of polarizations, namely, the \emph{real} polarizations (i.e. $\mathcal{P} = \bar{\mathcal{P}}$) and the \emph{complex} polarizations (i.e. $\mathcal{P} \cap \bar{\mathcal{P}} = \{0\}$). Moreover, in the event that $(M,\omega)$ is equipped with a compatible integrable almost complex structure, then the holomorphic tangent bundle $T^{1,0}M$ itself is a complex polarization, called the \emph{Kähler polarization}.

\begin{rmk}
The quantum space associated with a real, or mixed type, polarization, may consist of distributional (therefore non-smooth) sections when the polarization has non-simply-connected 
fibers which bring the need of imposing Bohr-Sommerfeld conditions.
\end{rmk}
The \emph{quantum Hilbert space} (associated to $\mathcal{P}$) is then defined as the closure of the space of polarized sections in the prequantum Hilbert space which, moreover, are required to satisfy some appropriate $L^2$ condition. Note that the quantum Hilbert space for a Kähler polarization would be given by the sections $s$ such that
\begin{align}\nonumber
    \nabla_{\frac{\partial}{\partial \bar z_j}}s = 0, 
\end{align}
for local holomorphic coordinates $(z_1, \dots, z_n)$ on $M$, 
which corresponds to the space holomorphic sections of $H^0(M,L)$.
Usually, the above framework is improved by including the so-called \emph{half-form correction}.
In the Kähler case, one introduces the canonical bundle
\begin{align}\nonumber
    K_M = \bigwedge^n(T^*M)^{1,0},
\end{align}
that is, the bundle whose sections are $(n,0)$-forms (and thus dependent on the complex structure). There is a natural metric on $K_M$ induced by the Kähler structure of $M$: if $\eta$ is an $(n,0)$-form, then
\begin{align}\nonumber
    \lVert \eta \rVert^2_{K_M} = \frac{\eta \wedge \bar\eta}{(2i)^n(-1)^{n(n+1)/2}\omega^n/n!}.
\end{align}
The half-form corrected quantum Hilbert space associated to a given Kähler polarization $\mathcal{P}$ is then given by the polarized sections of $L \otimes \sqrt{K_M}$, where $\sqrt{K_M}$ is a square root of $K_M$, that is, $\sqrt{K_M} \otimes \sqrt{K_M} \cong K_M$, with an appropriate connection on $\sqrt{K_M}$ induced by that of $K_M$. However, since $c_1(M) = -c_1(K_M)$, then this is possible if and only if $c_1(M)$ is even, which is not always the case. The line bundle whose sections comprise the Hilbert space would then satisfy
\begin{align}\nonumber
    c_1(L\otimes\sqrt{K_M}) = \left[\frac{\omega}{2\pi}\right] - \frac{c_1(M)}{2}.
\end{align}
This motivates us to consider instead, for the quantization procedure, a line bundle $L \to M$ with a different integrality condition, namely
\begin{align}\label{eq:half-form-integrality-condition}
    c_1(L) = \left[\frac{\omega}{2\pi}\right] - \frac{c_1(M)}{2} \in \mathbb{Z}.
\end{align}

\subsection{Toric K\"ahler geometry}\label{subsec_torickahlergeom}

Let $P$ be a Delzant polytope and $(M,\omega_P)$ the associated symplectic toric manifold with moment map $\mu$. The points of $M$ where the torus action is free make up a dense open subset $\mathring{M} = \mu^{-1}(\mathring{P})$, where $\mathring{P}$ is the interior of $P$. Moreover, $\mathring{M} \cong \mathring{P} \times T^n$ as symplectic manifolds, the latter with the symplectic structure inherited from $\mathbb{R}^n \times T^n$. 
One can describe toric Kähler structures with the aid of the so-called \emph{symplectic potential} \cite{Ab2,Gui2}. 
Denoting by $(x,\theta)$ the action-angle coordinates for $\mathring{P} \times T^n$, let
$P$ be given by the linear inequalities
\begin{align}\label{polytopeconditions}
l_r(x) = \langle x, \nu_r\rangle +\lambda_r\geq 0, \, r=1\, \dots, d,
\end{align}
where $\nu_j$ is the primitive inward pointing normal to the $jth$ facet of $P$. 
The function
\begin{align}\label{dfn:symp-potential}
    g_P(x) = \frac{1}{2}\sum_{r=1}^d l_r(x)\log l_r(x),
\end{align}
defined for $x \in \mathring{P}$, endows $M$ with a torus-invariant compatible complex structure $J_P$ given by the block matrix
\begin{align}\nonumber
    J_P =
    \begin{bmatrix}
    0 & -G_P^{-1} \\
    G_P & 0
    \end{bmatrix},
\end{align}
where $G_P$ is the Hessian of $g_P$ \cite{Gui2}. This is not the only compatible toric complex structure as shown by Abreu in \cite{Ab2}. More precisely,
\begin{thm}[\cite{Ab2}]
Let $(M,\omega)$ be the symplectic toric manifold associated to a given Delzant polytope $P$. If $J$ is a compatible toric complex structure, then
\begin{align}\label{eq:cpx-structure-symp-potential}
    J =
    \begin{bmatrix}
    0 & -G^{-1} \\
    G & 0
    \end{bmatrix}
\end{align}
where $G$ is the Hessian of
\begin{align}\label{eq:symp-potential-param}
    g = g_P + \phi,
\end{align}
$g_P$ being defined as in \eqref{dfn:symp-potential}, $\phi \in C^\infty(P)$. Moreover, $G$ is positive-definite on $\mathring{P}$ and satisfies the regularity condition
\begin{align}\nonumber
    \det G = \left(\delta(x)\prod_{r=1}^d l_r(x)\right)^{-1},
\end{align}
where $\delta$ is a smooth and strictly positive function on $P$.

Conversely, given a function $g$ as in \eqref{eq:symp-potential-param} with the above hypotheses, then $J$ as defined in \eqref{eq:cpx-structure-symp-potential} defines a compatible toric complex structure on $(M,\omega)$.
\end{thm}
A function $g$ as above is called a \emph{symplectic potential}.
(There is, in addition, a coordinate chart associated to each vertex $v$ of $P$, whose construction is detailed, for instance, in \cite{KMN1}.)

Let $J$ be a compatible complex structure induced by a symplectic potential $g$ on $(M,\omega)$. Then $(\mathring{M},J) \cong \mathring{P} \times T^n$ as Kähler manifolds, the latter with the complex structure induced from $\mathbb{C}^n$, via the  biholomorphism
\begin{align}\label{dfn:biholomorphism}
    \mathring{P} \times T^n &\xrightarrow{\cong} (\mathbb{C}^*)^n \\
    (x,\theta) &\mapsto w = (e^{y_1+i\theta_1}, \dots, e^{y_n+i\theta_n}),
\end{align}
where $y_j = \partial g/\partial x^j$. The assignment $x \mapsto y = \partial g/\partial x$ is an invertible Legendre transform, with inverse given by $x = \partial \kappa/\partial y$, where $$\kappa = x(y)\cdot y - g(x(y))$$ is a Kähler potential given in terms of $g$. 
Associated to $P$ there is the line bundle $L = \mathcal{O}(D) \to M$, where
\begin{align}\nonumber
D = \sum_{r=1}^d \lambda_rD_r,
\end{align}
with each $\lambda_r$ a non-negative integer and each $D_r$ being the divisor corresponding to the inverse image by the moment map of the facet of $P$ defined by $\ell_r = 0$. With this in mind, the meromorphic section $\sigma_D$ of $L$ whose divisor is given by $D$ trivializes $L$ over $\mathring{M}$. Moreover, the divisor of the meromorphic function, whose expression in the dense open subset is given by
\begin{align}\label{eq:monomials}
    w^m = w_1^{m_1}\cdots w_n^{m_n}, \quad m = (m_1, \dots, m_n) \in \mathbb{Z}^n,
\end{align}
can be calculated to be \cite{CJH}
\begin{align}\nonumber
    \mathrm{div}(w^m) = \sum_{r=1}^d \langle m, \nu_r \rangle D_r.
\end{align}
Now, since the holomorphic sections of $L$ are generated by the sections which, under the trivialization $\sigma_D$, are given by monomials as in \eqref{eq:monomials} with effective divisors, i.e.
\begin{align}\label{eq:holomorphic-sections-integers}
    H^0(M,L) = \mathrm{span}_{\mathbb{C}}\{w^m \sigma_D : m \in \mathbb{Z}^n, \mathrm{div}(w^m\sigma_D) \geq 0\},
\end{align}
and $\mathrm{div}(w^m\sigma_D) \geq 0$ precisely when $l_r(m) = \langle m, \nu_r \rangle - \lambda_r \geq 0$ for $r = 1, \dots, d$, it follows that there is a correspondence between the toric basis of $H^0(M,L)$ and the integral points of $P$.

\subsection{Half-form corrected quantization in the toric case}
\label{corrq}

Let $(M,\omega, I)$ be a compact smooth toric K\"ahler manifold with toric complex
structure $I$ and
such
that $\left[  \frac{\omega}{2\pi}\right]  -\frac{1}{2}c_{1}(M)$ is an ample
integral cohomology class. Let $K_{I}$ be the canonical line bundle on $M$.
Let the moment polytope be
\begin{align}\label{t11}
P = \left\{x \in \R^n \ : \ \ell_r(x) = \nu_r \cdot x + \lambda_r \geq 0, \ \ j = 1, \dots, d\right\} ,
\end{align}
where we use the freedom of translating the moment polytope to choose
the  $\{\lambda_j\}_{j=1,\dots,d}$ to be half-integral and defined as follows.
We consider an equivariant complex line bundle $L\cong {\mathcal O}(\lambda^L_1 D_1+\cdots+\lambda^L_dD_d)$ as in \cite{KMN1} such that
$c_{1}(L) = \left[\frac{\omega}{2\pi}\right]  -\frac{1}{2}c_{1}(M)$, where the $\{\lambda^{L}%
_{j}\}_{j=1,\cdots, d}$ define a polytope with integral vertices, $P_{L}$.
The half-integral
$\{\lambda_j\}_{j=1,\dots,d}$ in (\ref{t11}) are then defined by
\begin{align}\label{halflambdas}
\lambda_j := \lambda^{L}_{j}+\frac12, ~j=1,\dots, d,
\end{align}
in accordance with the fact that ${\rm div}\,(dZ)= - D_1 \cdots -D_d$.
Within the open orbit $U_{0}$, the holomorphic $(n,0)$-form $dz$ is given by:
$dZ=dz_1\wedge \cdots \wedge dz_n = dW/w^{\mathbf{1}}$, with
\begin{equation}\label{partialz}
z=\,^{t}(z_{1},\dots,z_{n})=\partial g/\partial{x}+i{\theta}
\end{equation}
and $dW:=dw^{1}\wedge\cdots\wedge dw^{n}$, such that $dZ$ and $dW$ are trivializing sections of
${K_{M}}_{|_{U_{0}}}$. (Here, $\mathbf{1}=(1,\dots,1)$ so that
$w^{\mathbf{1}}=w^1\cdots w^n$.)
 Note that $P_L$ is obtained from
the moment polytope $P$ by
shifts of $\frac12$ along each of the integral primitive inward
pointing normals.
We will call $P_{L}\subset P$ the \textit{corrected polytope}.
 We equip the line bundle $L$ with a $U(1)$
connection, $\nabla^{I}$, whose curvature is given by 
\begin{align}
F_{\nabla^{I}}= -i\omega+\frac{i}
{2}\rho_{I}.
\end{align}
From the analysis in the previous section, we recognize that $\sqrt{|K_I|}$ and $\mu_I$ 
are always well-defined, even when $\sqrt{K_I}$ is
not. Based on this, the authors of \cite{KMN1} defined the quantum space for half-form corrected K\"ahler quantization of $(M, \omega, L, I)$ as follows:

\begin{dfn}\label{qhs}\cite[Definition 3.1]{KMN1}
The quantum Hilbert space for the half-form corrected K\"ahler quantization of
$(M,\omega,L,I)$ is defined by
\[
{\mathcal{H}}^{Q}_{I} = \mathcal{B}_{I}^{Q} \otimes\mu_{I},
\]
where
\[
\mathcal{B}_{I}^{Q}=\{s\in\Gamma(M, L): \nabla^{I}_{\overline{\mathcal{P}}_{I}}.
s=0\}.
\]
The inner product is defined by
\begin{align}\label{innerproduct}
\langle \sigma\otimes \mu_I ,\sigma' \otimes \mu_I\rangle = \langle \sigma,\sigma'\rangle = \frac{1}{(2\pi)^n}\int_M h^L(\sigma,\sigma') \frac{\omega^n}{n!}.
\end{align}
\end{dfn}

Now fix a choice of symplectic potential $g$ for the complex structure $I$
on $M$. We define the connection $\nabla^{I}$ on $L$ by 
\begin{align}\label{conn1}
\Theta_{v}&=-i\,{x}_{v}\cdot d{\theta}_{v}+\frac{i}{2}\sum_{k=1}^n d\theta_{v}^{k}+\frac
{i}{4}\left(  \frac{\partial}{\partial x_{v}}\log\det  G_{v}\right)  \cdot
G_{v}^{-1}d\theta_{v}\\
&  =-i\,{x}_{v}\cdot d{\theta}_{v}+\frac{i}{2}\mathrm{Im}\left(  \partial\log\det
G_{v}+\sum_{k=1}^{n}dz_{v}^{k}\right)=\frac{\nabla^{I}{\mathbf{1}}_{v}^{U(1)}}{{\mathbf{1}}_{v}^{U(1)}%
}. \nonumber
\end{align}

On the open orbit $\mathring{M}$, the connection forms are specified by
\begin{align}\label{conn2}
\Theta_{0}  &  :=-i\,{x}\cdot d{\theta}+\frac{i}{4}\left(  \frac{\partial
}{\partial x}\log\det G\right)  \cdot  G^{-1}d\theta\\
&  =-i\,{x}\cdot d{\theta+}\frac{i}{2}\mathrm{Im}\partial\log\det G.\nonumber
\end{align}
One may check that $\Theta_{v}-\Theta_{v^{\prime}}=d\log\tilde{g}_{v^{\prime
}v}^{L}$ and $\Theta_{v}-\Theta_{0}=d\log\tilde{g}_{0v}^{L}$ so that
$\{\Theta_{0},\Theta_{v}: v\in V\}$ does indeed define a $U(1)$-connection on $L$.

\begin{rmk}
Despite the singularity of $d\theta_v^j$ as $x_v^j \to 0$, the connection form $\Theta_v$ defined in (\ref{conn1}) remains non-singular over the open set $U_v$. This non-singularity can be verified either by analyzing the behavior of the matrix $G_v$ or by using the alternative coordinates $\{a_v^j, b_v^j\}_{j=1,\dots,n}$ which regularize the expression.
\end{rmk}

\subsection{Complex-time dynamics and quantization}\label{section:cpx-time}
We shall collect here the results in \cite{BFMN, MN, SZ} concerning the imaginary time flow formalism in the concrete case of symplectic toric manifolds. Recall that if $(M,\omega_0,J_0)$ is a compact K\"ahler manifold with K\"ahler form $\omega_0$, the space of K\"ahler metrics on $M$ in the same K\"ahler class is given in terms of global relative K\"ahler potentials as 
\begin{align}\nonumber
    \mathcal{H}(\omega_0) = \{\rho \in C^{\infty}(M) \mid \omega_0 + i\partial\bar\partial\rho > 0\}/\mathbb{R},
\end{align}
which has a natural infinite-dimensional smooth manifold structure as it is an open subset in $C^\infty(M)$ (modulo constants), the latter considered with the topology of uniform convergence on compact sets of the functions and their derivatives of all orders. Moreover, from this fact, the tangent vectors are simply functions on $M$. This space is equipped with the so-called \emph{Mabuchi metric} \cite{M}, given by
\begin{align}\nonumber
    \langle \delta_1 \rho, \delta_2 \rho \rangle = \int_M\frac{1}{n!}(\delta_1\rho\cdot\delta_2\rho) \omega_\rho^n,
\end{align}
where the functions $\delta_1 \rho, \delta_2 \rho\in C^\infty(M)$ are tangent vectors at $\rho \in \mathcal{H}(\omega_0)$ and $\omega_\rho = \omega_0 + i\partial\bar\partial\rho$. From Moser's theorem, there are diffeomorphisms defining equivalent K\"ahler structures $\varphi_\rho:(M,\omega_\rho,J_0)\to (M,\omega_0,J_\rho)$, so that the K\"ahler metrics can then be described in terms of a varying complex structure $J_\rho$ for fixed symplectic form $\omega_0$. 

Let now $(M,\omega)$ be the symplectic toric manifold defined by the Delzant polytope $P$ and let $h$ be a smooth uniformly convex function on $P$. It is known (see \cite{SZ}) that the family of symplectic potentials
$$
g_s = g + sh, s \geq 0,
$$
defines a Mabuchi geodesic ray of toric K\"ahler structures. Let $\mathcal{P}_s$ denote the K\"ahler polarization defined by the complex structure $J_s$ on $(M,\omega_P)$ by the symplectic potential $g_s$. One has from \cite{BFMN}, pointwise in the Lagrangian Grassmannian along $\mathring{M}_P$ , 
$$
\lim_{s\to +\infty} \mathcal{P}_s = \mathcal{P}_\mathbb{R}, 
$$
where $\mathcal{P}_\mathbb{R}$ is the toric real polarization defined by the fibers of the moment map $\mu.$
As dscribed above, let now $L \to M$ be the smooth line bundle with first Chern class
\begin{align}\nonumber
    c_1(L) = \frac{\omega}{2\pi} - \frac{c_1(M)}{2}.
\end{align}
From \cite{KMN1}, we know that the half-form corrected Hilbert space $\mathcal{H}_{\mathcal{P}_{s}}$ associated to the polarization $\mathcal{P}_{s}$ is given by
\begin{align}\label{kahlerquants}
    \mathcal{H}_{\mathcal{P}_{s}} = \mathrm{Span}_{\mathbb{C}}\left\{\sigma_s^m = w_s^m e^{-k_s(x)}\mathbf{1}^{U(1)} \otimes \sqrt{dZ_s}\right\}, \text{ }m \in P \cap \mathbb{Z}^p,
\end{align}
where $$k_s(x) = x \cdot \frac{\partial g_s}{\partial x} - g_s(x),$$ is the K\"ahler potential, $w_s^j = e^{y^j_s + i\theta_j}$, $y_s^j =  \frac{\partial g_s}{\partial x_j}$, and 
$\mathbf{1}^{U(1)}$ is a local trivializing section for $l$, so that 
$\mathbf{1}^{U(1)} \otimes \sqrt{dZ_s}$ is a trivialization of $L$ over the open dense subset $\mathring M$. While $L$ is not necessarily the tensor product of a prequantum line bundle with a square root of the canonical bundle of $M$, we write local sections in this way as their product is well-defined (cf. \cite{Wo}, \S 10.4).

{}

%It is then a natural question to ask how sections $\sigma_s^m$ and $\sigma_s^{m'}$ %corresponding to different polarizations $\mathcal{P}_{is}$ and $\mathcal{P}_{is'}$, %respectively, relate to each other. 

Let $\hat h^{prQ}$ be the prequantum operator associated to $h$. 
From \cite{KMN1,KMN4}, we have a linear $T^n -$equivariant isomorphism 
$$
e^{s\hat h^{prQ}} \otimes e^{is\shL_{X_H}} : \mathcal{H}_{\mathcal{P}_{t}} \to \mathcal{H}_{\mathcal{P}_{t+s}}, t, s\geq 0,
$$
with
\begin{align}\label{eq:prq-evolution}
   e^{s\hat h^{prQ}} \otimes e^{is\shL_{X_H}}   \sigma_t^m = \sigma_{t+s}^m, \text{ }m \in P \cap \mathbb{Z}^n.
\end{align}
 
 Following \cite{KMN1,KMN4}, we define the quantum operator $h^Q:\mathcal{H}_{\mathcal{P}_{t}}\to \mathcal{H}_{\mathcal{P}_{t}}$ associated to $h$ by
 $$
 h^Q \sigma_t^m = h(m) \sigma_t^m, \, m\in P \cap \mathbb{Z}^n.
 $$
 
 The generalized coherent state transform $U_s, s>0$, which lifts the imaginary Hamiltonian flow along the Mabuchi geodesics generated by $h$ to the bundle of half-form corrected polarized Hilbert spaces, is defined as the linear isomorphism
 \begin{align}\label{defcst}
    U_{s} = (e^{s\hat{h}^{prQ}} \otimes e^{is\shL_{X_H}}) \circ e^{-s\hat h^Q} : \mathcal{H}_{\mathcal{P}_t} \to \mathcal{H}_{\mathcal{P}_{t+s}}.
\end{align}
We now recall 

\begin{thm}\label{thmkmn2}\cite[Theorem 4.3]{KMN4}
 Let $\left\{\delta^m\right\}_{m\in P \cap \mathbb{Z}^n}$ be the Bohr-Sommerfeld basis of distributional sections for the half-form corrected Hilbert space of $\mathcal{P}_\mathbb{R}-$ polarized sections (see \cite{BFMN,KMN1}). Then,
 $$
 \lim_{s\to +\infty} U_s \sigma_0^m = (2\pi)^{n/2} e^{g(m)} \delta^m.
 $$
\end{thm}
 In the following we will generalize this theorem to the case when the hamiltonian function generating the Mabuchi geodesic ray of toric K\"ahler metrics is convex only along a subset of the moment map coordinates so that the limit polarization becomes a mixed polarization. 

\section{Quantization in the Fourier polarization $\mathcal{P}_\infty$}\label{sectionPinfty}

%section{Half-form corrected mixed quantization}

In this Section, we will describe the quantization of the toric variety $M$ along the Mabuchi geodesic of toric K\"ahler structures in Theorem \ref{prop:evol-polarization} and also the quantization with respect to the limit polarization $\mathcal{P}_\infty$ in Theorems \ref{thm_polsectionsinside} and \ref{convsections}.

\subsection{Geodesic rays converging to mixed polarizations}

\label{subsec_geodraysconv}

Let $T^{p}$ denote a subtorus of $T^{n}=\mathbb{R}^n/\mathbb{Z}^n$. In this subsection, we describe the mixed polarization (which we denote by $\mathcal{P}_\infty$) associated with the action of the subtorus $T^{p}$ and a toric complex structure $J$ determined by a symplectic potential $g$. $\mathcal{P}_\infty$ is a Fourier polarization, in the sense described in the Introduction and in Section \ref{newsubsec}.

Associated to the subtorus $T^p$ we have a primitive sublattice of $\mathbb{Z}^n$, such that, infinitesimally, the Lie algebra 
inclusion  $\frak{t}^{p} \subset \frak{t}^{n}$ is represented by $p$ primitive vectors in $\mathbb{Z}^n$ generating the sublattice. Let $\hat B$ be the $(p\times n)$-matrix with those vectors 
in the rows. $\hat B$ is then the upper $(p\times n)$-block of a matrix $B \in SL(n,\mathbb{Z})$. (See Section 1 of \cite{AH} and Section 3 in Chapter 1 of \cite{GL}.)

The action of $T^{p}$ on $M$ is a Hamiltonian action, denoted by $\rho_{p}: T^{p} \rightarrow \mathrm{Diff}(M,\omega,J)$ and the corresponding moment map $\mu_{p}$ for this action is given by 
$$\mu_p=\hat B \circ \mu.$$ According to \cite{LW1} or \cite{W1},
the mixed (singular) polarization $\mathcal{P}_\infty$ can be defined as follows:
\begin{align}
\mathcal{P}_\infty = (\mathcal{D}_{\mathbb{C}} \cap P_{J} ) \oplus \mathcal{I}_{\mathbb{C}},
\end{align}
where $\mathcal{D}_{\mathbb{C}}=(\ker d \mu_{p} \otimes \CC)$ and $\mathcal{I}_{\mathbb{C}}=(\mathrm{Im} d \rho_{p} \otimes \CC)$. We denote the moment polytope for the Hamiltonian $T^p$ action by $\Delta = \mu_p(M).$

We will consider Mabuchi geodesic rays of toric K\"ahler structures generated by Hamiltonian functions which are smooth convex functions on $\Delta$. As we will describe, Theorem 1.2 
in \cite{BFMN} generalizes and one obtains a mixed toric polarization at infinite Mabuchi geodesic time (refer to \cite{BFMN,LW1,W1}). 

%REFER Augusto+wang et al  appropriately

 Consider the affine change of action-angle coordinates on $\mathring M$,

 $$\tilde x = B\cdot x, \,\,\tilde \theta = {}^t{B^{-1}} \cdot \theta,$$ 
    %be the corresponding new action-angle coordinates on $\mathring M$, 
    such that the moment polytope $P$ for the $T^n$-action is now described by the linear inequalities
    $$
    \tilde l_j (\tilde x) = \langle \tilde x, \tilde \nu_j \rangle - \lambda_j \geq 0,\,\,\, j=1, \dots r,
    $$
where $\tilde \nu _j = {}^tB^{-1} \nu_j$ gives the primitive normals to the facets of $P$ in the new affine coordinates. Note that one has, for the holomorphic coordinates on $\mathring M$ associated to a symplectic potential $g$, 
$$
\tilde z_j = \frac{\partial g}{\partial \tilde x_j} + i\tilde \theta_j = 
\sum_{k=1}^n B^{-1}_{kj} z_k,\,\, j=1, \dots, n,
$$
where $$z_k= \frac{\partial g}{\partial x_k} + i \theta_k,\, k=1, \dots, n.$$
Let 
\begin{align}\label{simplehamiltonian}
H = \frac12 \sum_{j=1}^p \tilde x_j^2 = \frac12 \vert\vert\mu_p\vert\vert^2.    
\end{align}

By the same argument as in the proof of Theorem 1.2 in \cite{BFMN}, Theorem 3.20 in \cite{LW1} and Theorem 3.5 in \cite{W1} we obtain the following theorem where, for simplicity, we present the case $H=\frac12 \vert\vert\mu_p\vert\vert^2$, noting that more general strictly convex functions of $\mu_p$  can be handled  by the same argument (see Remark \ref{remark42}).

\begin{thm}\label{prop:evol-polarization}
       For a choice of symplectic potential on $M$, $g$,    
     let $\mathcal{P}_s, s\geq 0$, be the family of K\"ahler polarizations associated to the symplectic potential $g+sH, s\geq 0$, which is obtained under the imaginary time flow of the Hamiltonian vector field $X_H$. Then the limiting polarization $\mathcal{P}_\infty$ exists, and
     $$\mathcal{P}_\infty:=  \lim_{s\to +\infty} \mathcal{P}_s $$ is a (singular) mixed polarization.
     Moreover, over the open dense $(\mathbb{C}^*)^n$-orbit $\mathring M,$ 
     $$
     \mathcal{P}_\infty = \langle  \partial/\partial\tilde \theta_{1}, \dots, \partial/\partial\tilde \theta_p, 
     X_{\bar{\tilde z}_{p+1}} , \dots, X_{\bar{\tilde z}_{n}},\rangle_{\mathbb C}= 
     $$
     $$
     =
     \langle  \partial/\partial\tilde \theta_{1}, \dots, \partial/\partial\tilde \theta_p, 
     \frac{\partial}{\partial {\tilde z_{p+1}}}, \dots, \frac{\partial}{\partial{\tilde z_{n}}}\rangle_{\mathbb C}.
     $$
\end{thm}

{}

\begin{proof}
Note that $\mathcal{P}_s, s\geq 0$ is generated on the open dense orbit by the 
Hamiltonian vector fields
$$
X_{\bar{\tilde{z}}^j_s}, j=1, \dots n.
$$
But we have 
$$
\tilde z^j_s = \tilde z^j_0 + s \tilde x^j,\,  j=1, \dots p,
$$
and
$$
\tilde z^j_s = \tilde z^j_0, \, j= p+1, \dots n,
$$
from which the result for the open dense orbit follows by taking $s\to \infty$. The second expression for $\mathcal{P}_\infty$ follows directly from the formulation in \cite{LW1} and the equivalence between the two expressions can also be directly checked by writing the Hamiltonian vector fields explicitly in action-angle coordinates and then taking the limit $s\to \infty.$
\end{proof}

We will also describe the explicit expressions for $P_\infty$ along $\mu^{-1}(\partial P)$ below in Proposition \ref{prop_limippolontheboundary}.
Thus, the limit polarization $\mathcal{P}_\infty$ has $p$ real directions and $n-p$ holomorphic directions. (Note that when $p=n$ one obtains the real toric polarization on $\mathring M$.)

Recall that around each vertex $v$ of $P$ one has a coordinate neighborhood $U_v$ and  holomorphic coordinates 
$(w_v^1, \dots, w_v^n)$ which are rational functions of the holomorphic coordinates $(w^1, \dots , w^n)$ along $\mathring M \cap U_v.$ If $A_v\in GL_n(\mathbb{Z})$ is matrix whose rows correspond to the primitive inner pointing normals to the facets of $P$ which are adjacent to $v$, recall that
one has affine coordinates on $U_v$
$$
x_v = A_v x + \lambda_v, \,\,\, \theta_v = {}^tA_v^{-1}\theta,
$$
where $\lambda_v$ is the $n\times 1$ matrix containing the $\lambda_j's$ corresponding the facets adjacent to $P$ given by the condition $l_j(x)=0$. One obtains,
$$
z_v = \frac{\partial g}{\partial x_v} + i \theta_v = {}^tA_v^{-1} z,
$$
and
$$
w_v^j = e^{z_v^j} = \Pi_{k=1}^n (w^k)^{(A_v^{-1})_{kj}},
$$
which is simply written as $w = w_v^{A_v}$. In terms of the new affine coordinates $(\tilde x, \tilde \theta)$ we obtain, correspondingly,
$$
\tilde w = \tilde w_v^{\tilde A_v},
$$
where $\tilde A_v = A_v B^{-1}.$ For the symplectic potentials $g_s = g+s H, s\geq 0$, consider the Hessian, with respect to the new action coordinates $\tilde x$,
$$
\tilde G_s = \tilde G + s T,
$$
where $T$ is the $n\times n$ matrix whose top $p\times p$ diagonal block is the identity and whose remaining entries are zero and where $\tilde G$ is the Hessian for $g$. Let $D$ denote the lower $(n-p)\times (n-p)$ diagonal block in $\tilde G$.

As in Theorem 3.4 in \cite{BFMN}, in any of the charts $(U_v,\tilde w_{v})$, the limit polarization 
$\mathcal{P}_{\infty}$ can be described as follows: for any face $F$ in the coordinate neighbourhood, we write abusively $j \in F$ if $\tilde w_{v}^{j}=0$ along $F$. Over the boundary of the moment polytope one then obtains

\begin{prop}
\label{prop_limippolontheboundary}

Over $\mu^{-1}(F) \cap U_v$,

    \begin{eqnarray}\nonumber
     && (\mathcal{P}_\infty) :=
     \lim_{s\to +\infty} \left( \langle \frac{\partial}{\partial \tilde w_v^{j}}, 
     j\in F \rangle_{\mathbb{C}} \oplus \langle \sum_{k\notin F} (\tilde G_s^v)^{-1}_{jk}\frac{\partial}{\partial \tilde x_v^k} -i\frac{\partial}{\partial \tilde \theta_v^j} , \, j\notin F \rangle_{\mathbb{C}}\right) 
     \\ \nonumber
     && = \langle \frac{\partial}{\partial \tilde w_v^{j}}, 
     j\in F \rangle_{\mathbb{C}} \oplus \langle \sum_{k\notin F, q,r=1}^{n} (\tilde{A}_{v})_{jq} (\tilde G^{-1}_\infty)_{qr}(\tilde{A}^{T}_{v})_{rk} \frac{\partial}{\partial \tilde x_v^k} -i\frac{\partial}{\partial \tilde \theta_v^j} , \, j\notin F \rangle_{\mathbb{C}}
     \\ \nonumber
     && =  \langle \frac{\partial}{\partial \tilde w_v^{j}}, 
     j\in F \rangle_{\mathbb{C}} \oplus \langle
     \sum_{q,l=p+1}^{n} (\tilde A_v)_{jq} D^{-1}_{(q-p)(l-p)} \frac{\partial}{\partial \tilde x^l_v}-i \frac{\partial}{\partial \tilde \theta_v^j},\, j\notin F\rangle_\mathbb{C},
    \end{eqnarray}
    
     where $\tilde G_s^v$ is the Hessian of $g_s$ with respect to the coordinates $\tilde x_v.$
\end{prop}

\begin{proof}
    The result follows from Lemma 3.3 in \cite{BFMN} and from the form of the inverse of the Hessian of $g_s$ as $s\to \infty.$ Note that if 
    $$
    \tilde G = \left[ 
    \begin{array}{cc}
        A_1 & A_2 \\
        A_3 & D
    \end{array}
    \right],
    $$
where $A$ is the top diagonal $p\times p$ block in $\tilde G$, then
$$
\tilde G_s = \left[ 
    \begin{array}{cc}
        A_1 + sI_p & A_2 \\
        A_3 & D
    \end{array}
    \right].
$$

(Note that $A_2^T = A_3$ and that $A_1, D$ are symmetric.)
Define \( S = A_1 + sI_p - A_2D^{-1}A_3 \). It is clear that \( S \) becomes invertible when \( s \) is sufficiently large. Using the inverse formula, we obtain:
\[
\tilde{G}_{s}^{-1}= 
\begin{bmatrix}
S^{-1} & -S^{-1} A_{2}D^{-1} \\
-D^{-1} A_{3} S^{-1} & D^{-1} + D^{-1} A_{3} S^{-1} A_{2} D^{-1}
\end{bmatrix}.
\] 
so that 
$$
\tilde{G}_{\infty}^{-1}=\lim_{s\to \infty} \tilde G_s^{-1} =
\left[ 
    \begin{array}{cc}
        0  & 0 \\
        0 & D^{-1}
    \end{array}
    \right],
$$
from which the result follows.
\end{proof}
{}

\begin{exmp} 
Let us first consider the case of a $4$-dimensional toric manifold $M$ with moment polytope $P$. Consider the following component of the moment map, given in polytope coordinates by

\begin{align}\nonumber
\tilde x_1(x_1,x_2) = a_1x_1 + a_2x_2,
\end{align}
where $a_1$ and $a_2$ are coprime integers and $(x_1,x_2)\in P$. The complex coordinates $w_j = e^{z_j}$, with $z_j = y_j+i\theta_j$ and $y_j = \partial g/\partial x_j,\, j=1,2$, define a (Kähler) polarization $\mathcal{P}$. 

Setting $H = (\tilde x_1)^2/2$, we then proceed to calculate the evolution $\mathcal{P}_{s}, s\geq 0$, of this polarization under the imaginary time flow of the Hamiltonian vector field $X_H$ in order to examine its behavior as $s \to \infty$. First, notice that $X_{w_j} = w_jX_{z_j}$, $j = 1, 2$, meaning that $\mathcal{P}$ is spanned by $X_{z_1}$ and $X_{z_2}$. Then
\begin{align}\nonumber
    e^{isX_H}\cdot z_j = y_j + i(\theta_j - isa_j\tilde x_1) = y_j + sa_j\tilde x_1 + i\theta_j,
\end{align}
and consequently,
\begin{align}\nonumber
    e^{isX_H}\cdot X_{z_j} = X_{y_j} + sX_{a_j\tilde x_1} + iX_{\theta_j} = s\left(X_{a_j\tilde x_1}+\frac{1}{s}(X_{y_j}+iX_{\theta_j})\right).
\end{align}
{}
We now introduce the new action coordinate
\begin{align}\label{eq:change-of-coords}\nonumber
    \tilde x_2 &= b_1x_1 + b_2x_2, \\ \nonumber
\end{align}
where $b_1$ and $b_2$ are integers such that $a_1b_2 - a_2b_1 = 1$. Their existence is guaranteed by the fact that $a_1$ and $a_2$ are coprime. In other words, we have $\tilde x = A\cdot x$ with 
\begin{align}\nonumber
A = \begin{bmatrix}
a_1 & a_2 \\
b_1 & b_2
\end{bmatrix}\in SL(2,\mathbb{Z}).
\end{align}
Note that $\tilde P = A(P)$ is an equivalent Delzant polytope describing the same toric manifold $M$.  The corresponding angle coordinates are $\tilde \theta = 
{}^tA^{-1}\cdot \theta$. These new action-angle coordinates also give rise to new complex coordinates compatible with the Kähler structure on $M$, given by $\tilde z_j = \tilde y_j + i\tilde\theta_j$, where $\tilde y_j = \partial g/\partial\tilde x_j$. In these new coordinates, the evolution is much simpler:
\begin{align}\nonumber
    e^{isX_H}\cdot \tilde z_2 &= \tilde z_2, \\ \nonumber
    e^{isX_H}\cdot \tilde z_1 &= \tilde z_1 + \tilde x_1 s,
\end{align}
and therefore
\begin{align} \nonumber
    e^{isX_H}\cdot X_{\tilde z_2} &= X_{\tilde z_2} \\ \nonumber
    e^{isX_H}\cdot X_{\tilde z_1} &= X_{\tilde z_1} + X_{\tilde x_1}s = s\left(X_{\tilde x_1} + \frac{1}{s}X_{\tilde z_1}\right),
\end{align}
and so at the limit $s \to \infty$ the polarization is spanned by $X_{\tilde z_2}$ and $X_{\tilde x_1} = -\partial/\partial\theta_1$.
\end{exmp}

Above,  we have used an $SL_n(\mathbb{Z})$ affine transformation so that  the Hamiltonian $H$ generating the Mabuchi family takes the simple form (\ref{simplehamiltonian}). 
For simplicity of notation, we will henceforth drop the tilde from the coordinates $\tilde x, \tilde \theta, \tilde z, \tilde w$ and from the matrices $\tilde A_v$, so that  they will be written simply as $x,\theta, z, w$ and $A_v$.

\subsection{Canonical line bundle associated to $\shP_{\infty}$}\label{sec-clb}

As $\shP_{\infty}$ is a complex Lagrangian distribution on $M$, it corresponds to a complex (singular) line bundle $K_{\infty}$ defined by
\begin{align}
    K_{\infty,p}=\{\alpha \in \wedge^{n}T^{*}_{p}M_{\CC}\mid \iota_{\bar{\xi}} \alpha =0, \forall \xi \in \shP_{p}\}.
\end{align}
Note that $K_{\infty}$ is a singular complex line bundle with mild singularity, that is, $K_{\infty}$ is smooth on the open dense subset of $M$.
From Section \ref{chapter:mixed-polarizations}, we obtain that the fiber of $K_\infty$  
over $p\in \mathring M$ is
$$
(K_\infty)_p = \langle dX_1^p \wedge d Z_{p+1}^n \rangle_\mathbb{C},
$$
where
$$
dX_1^p= d x_1 \wedge \cdots \wedge  dx_p,\,\, dZ_{p+1}^n= d{{z}}_{p+1}\wedge \cdots \wedge d{{z}}_n.
$$
The section  $dX_1^p \wedge dZ_{p+1}^n$ is a trivializing section of $K_\infty$ over $\mathring M$ which has a divisor along $\partial P$. 
Note that the divisor of $dz_1 \wedge \cdots \wedge dz_n$  is $-D_1-\cdots -D_d$ while $dx^j$ goes to zero on the facet where $l_j=0.$
Moreover, from Proposition \ref{prop_limippolontheboundary}, we see that $K_\infty$ has no nontrivial smooth sections over 
$M\setminus \mathring M$. 

\subsection{Quantization and coherent state transforms}
\label{subsecronaldo}

{}
In this Section, following \cite{KMN4} as recalled in Theorem \ref{thmkmn2}, we use a generalized coherent state transform,  to describe how polarized sections evolve along the Mabuchi ray of K\"ahler polarizations $\mathcal{P}_s, s\geq 0$, considered 
in Theorem \ref{prop:evol-polarization}.

Recall the families of toric K\"ahler structures considered in Section \ref{chapter:mixed-polarizations} which were defined by the  symplectic potential:
\begin{align}
g_{s}:=g+sH, \,\, s>0,
\end{align}
where the Hamiltonian $H$ is defined in (\ref{simplehamiltonian}). (Recall that $\tilde x$ is denoted simply as $x$ in this Section.)
Denote the corresponding $s$-dependent complex structure by $I_{s}$.
 
As explained in \cite{KMN4}, and recalled in Theorem \ref{thmkmn2}, for the case when $p=n$ and $\mathcal{P}_\infty = \mathcal{P}_\mathbb{R}$, the evolution of polarized $I_s$-holomorphic sections along the Mabuchi geodesic is nicely described by a generalized coherent state transform $U_s$ in (\ref{defcst}). The same approach  applies successfully to the more general case that we are considering here with $p<n$, as we now describe. Recall that the quantization of $M$ with respect to the K\"ahler polarizations $\mathcal{P}_s, s>0$ is given by 
$$
\mathcal{H}_{\mathcal{P}_s} = \mathrm{span}_{\mathbb{C}}\left\{\sigma^m_s, m\in P\cap \mathbb{Z}\right\},
$$
where $\sigma^m_s$ is the monomial section $\sigma^m$, as defined in Section 
\ref{section:cpx-time}, with respect to the toric complex structure $I_s$. 
 We define the quantum operator corresponding to $H$,
 $\hat H^Q:\mathcal{H}_{{\mathcal P}_s}\to \mathcal{H}_{{\mathcal P}}$, to be given by
 $$
 \hat H^Q \sigma^m_s = H(m) \sigma^m_s, m\in P\cap {\mathbb Z}^n,\, s\geq 0.
 $$
 Note that $\hat{H}^Q$ is  given by
\begin{align}\nonumber
    \hat{H}^Q = \frac{1}{2}\sum_{j=1}^p (\hat{x}_{j}^{prQ})^2,
\end{align}
where $x_j^{prQ}$ is the Kostant-Souriau prequantum operator associated to $x_j.$
  The generalized coherent state transform (gCST) is then defined by the  $T^n-$equivariant linear isomorphism 
\begin{align}\label{gcst}
    U_{s} = (e^{s\hat{H}^{prQ}} \otimes e^{is \shL_{X_H}}) \circ e^{-s\hat H^Q} : \mathcal{H}_{\mathcal{P}_{0}} \to \mathcal{H}_{\mathcal{P}_{s}}, s \geq 0,
\end{align}
where, as before, $H^{prQ}$ is the Kostant-Souriau prequantum operator associated to $H$. 
It is immediate to verify the validity of the following analog of Theorem \ref{thmkmn2} for the case $p<n$:
\begin{thm}\label{local-description}For $s\geq 0$,
$$
U_{s} \sigma^m_{0} = e^{-s H(m)}\sigma^m_{s}, \, m\in P\cap \mathbb{Z}.
$$
\end{thm}
{}
{}
Also, as in the case $p=n, \mathcal{P}_\infty = \mathcal{P}_\mathbb{R},$ the gCST gives a well-defined limit for the quantization at infinite geodesic time along the Mabuchi ray. While in that case the $I_s$-holomorphic sections converge, as $s\to\infty$, to Dirac delta distibutional sections supported on Bohr-Sommerfeld fibers, which are the fibers with integral value of the moment map, in the more general present case of $p<n$ we have convergence to distributional sections where only $p$ of the moment map coordinates localize.
Consider the following distributional sections of $L$
\begin{align}
%\label{limitsections}
    \tilde{\sigma}^m_\infty := \sqrt{2\pi}^p\delta^{m_{1}^p}W^m   \sqrt{dX_{1}^p \wedge dZ_{p+1}^{n}},
\end{align}
where 
  $$\delta^{m^p_{1}} = \delta(x_{1}-m_{1})\cdots\delta(x_p-m_p)$$ 
  is the Dirac delta distribution, 
  and 
  $$W^m = e^{-k}w_{p+1}^{m_{p+1}}\cdots w_{n}^{m_{n}}e^{i\sum_{j=1}^p m_j \theta_j}e^{\sum_{j=1}^p m_j\frac{\partial g}{\partial x_{j} }}.$$

Let $\underline m = (m_1, \dots, m_p).$  
Note that, since $\kappa= x\cdot y -g$, with $y=\frac{\partial g}{\partial x}$, on the support of the product of delta functions, $x_j=m_j, j=1, \dots, p$, and  
$$
-\kappa +\sum_{j=1}^p m_j\frac{\partial g}{\partial x_{j} } = -\kappa_{\mathrm{red}},
$$
where $\kappa_\mathrm{red}$ is the K\"ahler potential obtained on the 
K\"ahler quotient  $M_{\underline m}=\mu_p^{-1}(m_1,\dots, m_p)/T^p$ and which is given by the Legendre transform of the restriction of $g$ to $\mu_p^{-1}(m_1, \dots, m_p)$. Indeed, as recalled in Section \ref{newsubsec} , from \cite{CDG}, this restriction gives the symplectic potential for the reduced K\"ahler structure on the K\"ahler quotient $M_{\underline m}.$

We have therefore that 
\begin{align}
\label{limitsections}
    \tilde{\sigma}^m_\infty := \sqrt{2\pi}^p\delta^{m_{1}^p}e^{i\sum_{j=1}^p m_j \theta_j}\sigma^{\underline m, m'}   \sqrt{dX_{1}^p},
\end{align}
where $m'=(m_{p+1}, \dots, m_n)$ and 
\begin{align}
\label{reducedsections}
\sigma^{\underline m, m'} = e^{-\kappa_\mathrm{red}} w_{p+1}^{m_{p+1}}\cdots w_{n}^{m_{n}} \sqrt{dZ_{p+1}^{n}}
\end{align}
are the monomial sections for the K\"ahler quantization of $M_{\underline m}$ with respect to the reduced K\"ahler structure for the toric K\"ahler structure defined by the symplectic potential $g$, which is the initial point of the Mabuchi ray of toric K\"ahler structures $g_s, s\geq0$. (See  also Section \ref{newsubsec}.)

Note further that, from Lemma 3.3 in \cite{KMN1}, $\tilde{\sigma}^m_\infty$ gives a well-defined distributional section on $M$. Indeed, $\tilde{\sigma}^m_\infty$ has no poles along $\mu^{-1}(\partial P)$ since the zeros of $w_{p+1}^{m_{p+1}}\cdots w_n^{m_n}$ cancel the poles of $\sqrt{dZ^{n}_{p+1}}$. 
{}

Following, the proof of Theorem 4.3 in \cite{KMN4}, we obtain for the $s \to +\infty$ limit,

\begin{thm}\label{convsections}
   Let $s\geq 0.$
  \begin{align}\label{eq:section-limit-cst}
      \lim_{s\to +\infty} U_{s}\sigma^m_{0} = \tilde{\sigma}^m_\infty, \, m\in P\cap \mathbb{Z}^n.
  \end{align}
  \end{thm}

{}

\begin{proof}
First we calculate the evolution of the half-form. The operator $e^{is \shL_{X_H}}$ acts on half-forms while keeping the consistency of its action on the algebra of smooth functions, that is,
\begin{align*}\nonumber
    &e^{is\shL_{X_H}}\sqrt{dZ} =\sqrt{ d\left(sx_{1} + z_{1}\right) \wedge \cdots \wedge d\left(sx_p +z_p \right)\wedge dZ_{p+1}^n }\\
    &= \sqrt{s}^p\sqrt{ d\left(x_{1} + \frac{1}{s}z_{1}\right) \wedge \cdots \wedge d\left(x_p + \frac{1}{s}z_p \right)\wedge dZ_{p+1}^n}.
\end{align*}
Now, using the  connection on $L$ given by
\begin{align}\nonumber
    \nabla \mathbf{1}^{U(1)} = -i\sum x_j d\theta_j\mathbf{1}^{U(1)},
\end{align}
a routine calculation gives us
\begin{align}\nonumber
    \hat{H}^{prQ} = -\frac{1}{2}\sum_{j=1}^p x_{j}^2 - i\sum_{j=1}^p x_{j}\frac{\partial}{\partial \theta_{j}}.
\end{align}
Furthermore,
\begin{align}\nonumber
    \hat{H}^Q = - \frac{1}{2}\sum_{j=1}^p \frac{\partial^2}{(\partial \theta_{j})^2}.
\end{align}
Thus,
\begin{align*}
    e^{s\hat{H}^{prQ}}\circ e^{-s{H}(m)} w^m = 
    e^{-s\alpha}e^{\beta}e^{i\Theta}w_{p+1}^{m_{p+1}}\cdots w_{n}^{m_{n}},
\end{align*}
where
\begin{align*}
    &\alpha(x_{1},\dots,x_p) = \sum_{j=1}^p \frac{1}{2}(x_j-m_j)^2, \\
    &\beta(x_{1},\dots,x_p) = \sum_{j=1}^p m_j\frac{\partial g_{0}}{\partial x_{j} }, \\
    &\Theta(\theta_{1},\dots,\theta_p) = \sum_{j=1}^p m_j\theta_j.
\end{align*}
We denote $W^{m}$ as $$W^m = e^{-k}w_{p+1}^{m_{p+1}}\cdots w_{n}^{m_{n}}e^{i\sum_{j=1}^p m_j \theta_j}e^{\sum_{j=1}^p m_j\frac{\partial g}{\partial x_{j} }}.$$ 
   We then have:
    \begin{align*}
   & \lim_{s\to +\infty} U_{s}\sigma^m_{0}=\lim_{s\to +\infty} e^{-s\alpha}W^{m}\sqrt{s}^p\sqrt{ dX_{1}^{P}\wedge dZ_{p+1}^n +O(\frac{1}{s})}\\
    &=\lim_{s\to +\infty}\left(\left(\frac{s}{2\pi}\right)^{p/2}e^{-s\sum_{j=1}^p \frac{1}{2}(x_j-m_j)^2}\right){\sqrt{2\pi}^p}W^{m}\sqrt{ dX_{1}^{P}\wedge dZ_{p+1}^n +O(\frac{1}{s})}\\
    &=\sqrt{2\pi}^p~\delta^{m_{1}^p}~W^m   \sqrt{dX_{1}^p \wedge dZ_{p+1}^{n}}= \tilde{\sigma}^m_\infty.
    \end{align*}
\end{proof}
%\subsection{Relation to symplectic reduction}\label{section-red}
In the next sections, we will show that the distributional sections $\tilde{\sigma}^m_\infty$ are $\mathcal{P}_\infty$-polarized and that, for $m\in P\cap \mathbb{Z}^n$, they span the vector space of $\mathcal{P}_\infty$-polarized sections, so that the gCST provides a continuous interpolation between the quantization along the Mabuchi ray from $s=0$ to $s=\infty$. 
In order to show that for the quantization wth respect to $\mathcal{P}_\infty$ no new distributional polarized sections arise with support along the boundary $\partial \Delta$, we will now recall in detail the study of half-form polarized sections 
in \cite{KMN1}.

\subsection{Quantum space $\mathcal{H}_{\mathcal{P}_\infty}$ for the half-form corrected $\mathcal{P}_\infty$}

In this Section, we prove that the Hilbert space of half-form corrected sections polarized with respect to 
$\mathcal{P}_\infty$ consists of the linear span of the basis $\left\{\tilde{\sigma}^m_\infty\right\}_{m\in P\cap \mathbb{Z}^n}$.

{}

For the family of complex structures $I_s, s\geq 0$ described above, consider the family of connections $\Theta_v$ described in (\ref{conn1}) and (\ref{conn2}) where 
$G_s = G + s T$, where $T$ was defined above Proposition \ref{prop_limippolontheboundary}. To define the condition for a half-form corrected section to be $\mathcal{P}_\infty$-polarized, we follow \cite{KMN1} and define a ``limit connection" by considering the $s\to \infty$-limit of the operators of covariant differentiation along $\mathcal{P}_\infty$. 
Note that, along $\mathring{M}$ we have
$$
X_{{z}^j_s} = -2i \sum_{k=1}^n (G_s)_{jk} \frac{\partial}{\partial \bar{z}^k_s} =
-2i \left( \frac{\partial}{\partial x_j} + i \sum_{k=1}^n (G_{s})_{jk} \frac{\partial}{\partial \theta_k}\right)
$$
and that 
$$
X_{{z}^j_s} = X_{{z}^j_0} -s \frac{\partial}{\partial \theta_j},\, j=1, \dots p,
$$
$$
X_{{z}^j_s} = X_{{z}^j_0}, \, j=p+1, \dots n.$$
The behaviour of the connection forms as $s\to \infty$ is described as follows.

\begin{lemma}\label{lemma_limitconn}
On the open orbit $\mathring M$, we obtain
\begin{align} 
\Theta_{0}^{\infty} :=-i\,{x}\cdot d{\theta}+\frac{i}{4}\sum_{j,k=p+1}^{n}\left(  \frac{\partial
}{\partial x_{j}}\log\det D\right)  \cdot  D^{-1}_{j-p,k-p}d\theta_{k},
\end{align}
 and on the vertex chart $U_v$
\begin{align*}
\Theta_{v}^{\infty} :=&-i\,{x}_{v}\cdot d{\theta}_{v}+\frac{i}{2}\sum_{k=1}^n d\theta_{v}^{k}\\
&+\frac{i}{4}\sum_{i,j=1,q,l=p+1}^{n}\left(  \frac{\partial
}{\partial x_{v}^{j}}\log\det D\right)  \cdot  (A_{v})_{jq}D^{-1}_{q-p,l-p}(A_{v}^{T})_{l,i}d\theta_{v}^{i}.
\end{align*}
\end{lemma}
\begin{proof}
Now fix a choice of symplectic potential $g$ for the complex structure $I$
on $X$. According to equations (\ref{conn1}) and (\ref{conn2}), we have the connection $\nabla^{I}$ on $L$ defined by 
\begin{align*}
\Theta_{v}  & :=\frac{\nabla^{I}{\mathbf{1}}_{v}^{U(1)}}{{\mathbf{1}}_{v}^{U(1)}%
} =-i\,{x}_{v}\cdot d{\theta}_{v}+\frac{i}{2}\sum_{k=1}^n d\theta_{v}^{k}+\frac
{i}{4}\left(  \frac{\partial}{\partial x_{v}}\log\det  G_{v}\right)  \cdot
G_{v}^{-1}d\theta_{v}\\
&  =-i\,{x}_{v}\cdot d{\theta}_{v}+\frac{i}{2}\mathrm{Im}\left(\partial\log\det
G_{v}+\sum_{k=1}^{n}dz_{v}^{k}\right).
\end{align*}
For any symplectic potential $g_{s}$ for the complex structure $I_{s}$ on $X$, On the open orbit $\check X$, the connection forms are specified by
\begin{align*}
\Theta_{0}^{s}  &  :=-i\,{x}\cdot d{\theta}+\frac{i}{4}\left(  \frac{\partial
}{\partial x}\log\det G_{s}\right)  \cdot  G_{s}^{-1}d\theta\\
&  =-i\,{x}\cdot d{\theta+}\frac{i}{2}\mathrm{Im}\partial\log\det  G_{s},
\end{align*}
where$$ G_s = \left[ 
    \begin{array}{cc}
        A_1 + sI_p & A_2 \\
        A_3 & D
    \end{array}
    \right],
$$ 
with $A$ being the top diagonal $p\times p$ block in $\tilde G$ and $D$ being the invertible $(n-p) \times (n-p)$ matrix.
Moreover, 
$\det G_{s} \sim s^{p}\det D+O(s^{p-1})$, and $\lim_{s\rightarrow \infty}G_{s}^{-1}= \left[ 
    \begin{array}{cc}
      0 & 0 \\
        0 & D^{-1}
    \end{array}
    \right].$
We therefore obtain:
\begin{align*}
\Theta_{0}^{\infty}  &  :=-i\,{x}\cdot d{\theta}+\frac{i}{4}\sum_{j,k=p+1}^{n}\left(  \frac{\partial
}{\partial x_{j}}\log\det D\right)  \cdot  D^{-1}_{j-p,k-p}d\theta_{k}.
\end{align*}
Similarly, on the vertex chart, we have:
\begin{align*}
&\Theta_{v}^{\infty} :=-i\,{x}_{v}\cdot d{\theta}_{v}+\frac{i}{2}\sum_{k=1}^n d\theta_{v}^{k}+\frac{i}{4}\left(  \frac{\partial
}{\partial x_{v}}\log\det D\right)  \cdot  A_{v}G_{\infty}^{-1}A_{v}^{T}d\theta_{v}\\
&=-i\,{x}_{v}\cdot d{\theta}_{v}+\frac{i}{2}\sum_{k=1}^n d\theta_{v}^{k}\\
&+\frac{i}{4}\sum_{i,j=1,q,l=p+1}^{n}\left(  \frac{\partial
}{\partial x_{v}^{j}}\log\det D\right)  \cdot  (A_{v})_{jq}D^{-1}_{q-p,l-p}(A_{v}^{T})_{l,i}d\theta_{v}^{i}.
\end{align*}
\end{proof}
We therefore define $B_{\mathcal{P}_{\infty}}$, the space of $\mathcal{P}_\infty$-polarized sections to given by the intersection of the kernels of the differential operators 
$$
\nabla^\infty_{\frac{\partial}{\partial \theta_j}} = \frac{\partial}{\partial \theta_j} + \Theta_0^\infty \left(\frac{\partial}{\partial \theta_j}\right), \, j=1, \dots p,
$$
and
$$
\nabla^\infty_{X_{z^j_0}} = X_{z^j_0} +  \Theta_0^\infty \left(X_{z^j_0}\right), \, j=p+1, \dots n.
$$

\begin{dfn}\label{correctedQS}
The quantum Hilbert space for the half-form corrected mixed polarization $\mathcal{P}_{\infty}$ is defined by 
$$
\mathcal{H}_{\mathcal{P}_{\infty}} =B_{\mathcal{P}_{\infty}} \otimes \sqrt{|dX_{1}^{p}\wedge dZ_{p+1}^{n}|},
$$
where $$B_{\mathcal{P}_{\infty}}=\{\sigma \in \Gamma(M, L^{-1})' \mid \nabla^{\infty}_{\xi} \sigma=0, \forall \xi \in \Gamma(M, \mathcal{P}_{\infty})\}.$$
\end{dfn}
Note that in terms of the vertex chart coordinates
$$
\frac{\partial}{\partial \theta_k} = \sum_{j=1}^n (A_v^{-1})_{jk} \frac{\partial}{\partial \theta_v^j}
$$
and
$$
X_{z^k_0} = \sum_{j=1}^n (A_v)_{jk} X_{z^j_{v,0}}.
$$
Therefore, on the vertex chart $U_v$ the conditions of $\mathcal{P}_\infty$-polarizability 
are given by the intersections of the kernels of
$$
\sum_{k=1}^n (A_v^{-1}){kj} \left(\frac{\partial}{\partial \theta_v^k} + \Theta_v^\infty\left(\frac{\partial}{\partial \theta_v^k}\right)\right), \, j=1, \dots, p,
$$
and
$$
\sum_{k=1}^n (A_v)_{kj} \left(X_{z^k_{v,0}} +  \Theta_v^\infty\left(X_{z^k_{v,0}} \right)  \right), 
j= p+1, \dots, n.
$$
{}

\begin{lemma} \label{outsidebdy}
There are no nonzero solutions of the equations of covariant constancy along $\mathcal{P}_\infty$ with support contained along $\mu_{p}^{-1}(\partial \Delta).$
\end{lemma}

\begin{proof}
We have the limit connection as before:
On the open orbit $\mathring M$, we obtain
\begin{align*} 
\Theta_{0}^{\infty} :=-i\,{x}\cdot d{\theta}+\frac{i}{4}\sum_{j,k=p+1}^{n}\left(  \frac{\partial
}{\partial x_{j}}\log\det D\right)  \cdot  D^{-1}_{j-p,k-p}d\theta_{k},
\end{align*}
 and on the vertex chart $U_v$
\begin{align*}
\Theta_{v}^{\infty} :=&-i\,{x}_{v}\cdot d{\theta}_{v}+\frac{i}{2}\sum_{k=1}^n d\theta_{v}^{k}\\
&+\frac{i}{4}\sum_{i,j=1,q,l=p+1}^{n}\left(  \frac{\partial
}{\partial x_{v}^{j}}\log\det D\right)  \cdot  (A_{v})_{jq}D^{-1}_{q-p,l-p}(A_{v}^{T})_{l,i}d\theta_{v}^{i}
\end{align*}

from the expression for $\Theta_v^\infty$ we see that, for $k=1, \dots, p,$ 
$$
\Theta^\infty_v \left( \frac{\partial}{\partial \theta_k}\right) =
-i \sum_{j=1}^n x_v^j (A_v^{-1})_{jk} + \frac{i}{2} \sum_{j=1}^n (A_v^{-1})_{jk}, 
$$
since the last term in the expression for $\Theta^\infty_v$ does not contribute.
Assume $\delta$ is the distributional section with support on $\mu_{-1}(\partial \Delta)$. Let us consider a solution with support along $\mu_{p}^{-1}(x_{v}^{j}=0)$, for some fixed $j=1, \cdots, p$. Here $(x_{v}^{1}, \cdots, x_{v}^{p})$ is a vertex coordinate for $\Delta$. Let $\check{u}=(u_{1},\cdots,u_{j-1},u_{j+1}, \cdots, u_{n})$ and $\check{v}=(v_{1},\cdots,v_{j-1},v_{j+1}, \cdots, v_{n})$. We take coordinates $(u_{j},v_{j},\check{u},\check{v})$ (see \cite{BFMN}), so that $x_{v}^{j}=0 \Leftrightarrow (u_{j},v_{j})=(0,0)$ and $$\nabla^{\infty}_{\frac{\partial}{\partial \theta_{v}^{j}}}=-i(-v_{j}\frac{\partial}{\partial  u_{j}}+u_{j}\frac{\partial}{\partial v_{j}}) + \frac{i}{2}$$ in that neighbourhood. The same argument as in the proof of \cite[Theorem 4.7]{BFMN}, $\nabla_{\frac{\partial}{\partial \theta_{v}^{j}}}^{\infty}\delta=0$ implies $\delta=0$. Therefore no nonzero solutions with support along $\mu_{p}^{-1}(\partial \Delta)$.
\end{proof}

\begin{thm}\label{thm_polsectionsinside}
The distributional sections $\tilde{\sigma}^m_\infty, m\in P\cap \mathbb{Z}^n$, in (\ref{limitsections}), are in $\mathcal{H}_{\mathcal{P}_\infty}$. Moreover, 
for any $\sigma \in \mathcal{H}_{\mathcal{P}_\infty}$, $\sigma$ is a linear combination of the sections $\tilde{\sigma}^m_\infty, m\in P\cap \mathbb{Z}^n$. Therefore, the distributional sections $\{\tilde{\sigma}^m_\infty, m\in P\cap \mathbb{Z}^n\}_{m\in P\cap \mathbb{Z}^n}$ form a basis of $\mathcal{H}_{\mathcal{P}_\infty}$.
\end{thm}

\begin{proof} 
Recall that, $\tilde{\sigma}^m_\infty$ on $\mathring M$ can be written as 
\begin{align}\label{local}
\tilde{\sigma}^m_\infty := \sqrt{2\pi}^p\delta^{m_{1}^p}W^m   \sqrt{dX_{1}^p \wedge dZ_{p+1}^{n}},
\end{align}
where 
  $$\delta^{m^p_{1}} = \delta(x_{1}-m_{1})\cdots\delta(x_p-m_p)$$ 
  is the Dirac delta distribution, 
  and 
  $$W^m = e^{-k}w_{p+1}^{m_{p+1}}\cdots w_{n}^{m_{n}}e^{i\sum_{j=1}^p m_j \theta_j}e^{\sum_{j=1}^p m_j\frac{\partial g}{\partial x_{j} }}.$$
Since $L\otimes\sqrt{|K_{\infty}|} \cong l \otimes K_{\infty}$ in $\mathring M$ and for any $\xi \in \Gamma(\mathring M, \mathcal{P}_{\infty})$, $\xi$ can be written as a linear combination of $\frac{\partial}{\partial \theta^{j}}$ and $\frac{\partial}{\partial \bar{z}_{k}}$, for $j=1,\cdots p$ and $k=p+1, \cdots, n$, we find that $\tilde{\sigma}^m_\infty$ belongs to $\mathcal{H}_{\mathcal{P}_\infty}$, by applying Equation (\ref{local}) and Lemma \ref{lemma_limitconn} $\tilde{\sigma}^m_\infty$.
 By Lemma \ref{outsidebdy}, we establish that there are no solutions supported on $\mu_{p}^{-1}(\partial \Delta)$. It remains to show that any $\sigma \in \mathcal{H}_{\mathcal{P}_\infty}$, $\sigma$ can be expressed as a linear combination of the sections $\tilde{\sigma}^m_\infty, m\in P\cap \mathbb{Z}^n$. The polarization $\mathcal{P}_\infty$ is mixed and over $\mathring M$ its real fibers are given by tori of dimension $p$ generated by the vector fields $\partial/\partial {\theta_j}, j=1, \dots, p.$ Therefore, distributional sections which are $\mathcal{P}_\infty$-polarized will be supported on the partial Bohr-Sommerfeld locus given by the conditions $(x_1, \dots, x_p)=(m_1, \dots ,m_p)\in \mathbb{Z}^p \cap \Delta$ where $\Delta$ is the projection of $P$ on the first $p$ momentum coordinates. The argument analogous to those in the proof of Theorem 4.7 in \cite{KMN1} or Proposition 3.1 in \cite{BFMN} then shows that the dependence of the polarized sections on the variables $(x_1, \dots, x_p)$ is given in terms of Dirac delta distributions and that, on the other hand, derivatives of Diract delta distributions derivatives are not covariantly constant. On the other hand, by Theorem \ref{prop:evol-polarization}, the remaining directions in $\mathcal{P}_\infty$ over $\mathring M$ are just given by 
the holomorphic vector fields $\frac{\partial}{\partial \bar{z}^j_0}, j=p+1, \dots, n$, so the dependence on the variables $(w_{p+1}, \dots , w_{n})$ is given by products of monomials times an exponential of minus the restriction of the K\"ahler potential to the partial Bohr-Sommerfeld  fiber, in the form of the term $W$ in (\ref{limitsections}).
Note also that the section $\sqrt{dX_{1}^p \wedge dZ_{p+1}^{n}}$ of the half-form bundle 
$K_{\mathcal{P}_\infty}$ is $\mathcal{P}_\infty$-polarized since the functions $x_1, \dots, x_p, z_{p+1}, \dots, z_{n}$ on $\mathring M$ are $\mathcal{P}_\infty$-polarized. Therefore, for any $\sigma \in \mathcal{H}_{\mathcal{P}_\infty}$, $\sigma$ is a linear combination of the sections $\tilde{\sigma}^m_\infty, m\in P\cap \mathbb{Z}^n$.
\end{proof}
%DeltaDelta

\begin{dfn}For any $\sigma_{0} \in \mathcal{H}_{\mathcal{P}_{0}}$, we define
$$U_{\infty}(\sigma_{0})=\lim_{s \to \infty} U_{s}(\sigma_{0}).$$ 
\end{dfn}
According to Theorem \ref{convsections} and Theorem \ref{thm_polsectionsinside}, we obtain the following result.
\begin{cor}\label{corQSconnection} 
Then, $U_{\infty}: \mathcal{H}_{\mathcal{P}_{0}} \rightarrow \mathcal{H}_{\mathcal{P}_\infty}$ is a $T$-invariant isomorphism. In particular,
$$U_{\infty}(\sigma_{0}^{m})=\lim_{s \to \infty} \tilde{\sigma}_{s}^{m}= \tilde{\sigma}^{m}_{\infty},$$
where $\{\sigma_{0}^{m}\}_{m\in P \cap \mathbb{Z}^{n}}$ is the canonical basis of $\mathcal{H}_{\mathcal{P}_{0}}$, and $\tilde{\sigma}_{s}^{m}=U_{s}(\sigma_{0}^{m})$.
\end{cor}

{}

\section{``Quantization commutes with Reduction" for toric manifolds: a new perspective}
\label{chapter:mixed-polarizations}
Let $M$ be a $n$--dimensional toric manifold and $T^p$ be a fixed torus subgroup of the toric group, $T^n$. The mixed limit polarization $\mathcal{P}_\infty$ described in the previous section is a Fourier polarization in the sense of Definition \ref{defFourierpolariz} and in fact, from the discussion in the proof of Theorem \ref{thm_polsectionsinside}, also in the sense of the more strict Definition 3.18 of \cite{BHKMN}. In this Section we describe aspects of the quantization commutes with reduction correspondence for the Fourier polarization $\mathcal{P}_\infty$ and also for the initial toric Ka\"hler polarization $\mathcal{P}_0.$

\subsection{``Quantization commutes with Reduction"}
\label{newsubsec}

In this section, we describe general features of appropriate toric mixed polarizations for which the 
discussion of the quantization commutes with reduction correspondence is naturally placed. These are associated to the action of a subtorus $T^p\subset T^n$, with moment map $\mu_p$ and moment polytope $\Delta.$ 
%We 
%call such polarizations ``Fourier polarizations".  (See Section 3 in \cite{BHKMN}.) In the next section, we will then describe explicitly 
%mixed toric polarizations $\mathcal{P}_\infty$ which are Fourier polarizations. 

We have the diagram
$$
M/T^p \stackrel{\pi}{\longleftarrow} M \stackrel{\mu_p}{\longrightarrow} \Delta ,
$$
where $\pi$ denotes the quotient map, and also the map $\alpha: M/T^p \to \Delta$ such that $\alpha \circ \pi = \mu_p.$

The fibers of $\mu_p$ are $T^p$ principal bundles in $M$ and that the real directions in 
$\mathcal{P}_\infty$ are given by the $T^p$ orbits. The fibers of $\alpha$ then 
give the coisotropic reductions for $\mathcal{P}_\infty$ and they correspond to the K\"ahler reductions 
$$
M_c = \mu_p^{-1}(c)/T^p
$$
with K\"ahler structure given by the reduction of the K\"ahler structure on $M$ defined by the symplectic potential $g$. (See Theorem \ref{prop:evol-polarization} where this is explicit.) This is an example of a ``fibering polarizations", as considered in Section 3 of \cite{BHKMN} (see also \cite{BFHMN}), more precisely $\mathcal{P}_\infty$ is a ``Fourier polarization", as in Definition 3.18 in \cite{BHKMN}. 

As will be seen quite explicitly below in Theorem \ref{prop:evol-polarization}, the toric K\"ahler structures along the coisotropic reductions of $\mathcal{P}_\infty$ do not evolve along the Mabuchi ray, so that, at infinite geodesic time, one will obtain that the quantizations of the coisotropic reductions of $\mathcal{P}_\infty$ coincide with the K\"ahler quantizations of the K\"ahler reductions of the initial K\"ahler structure. (See Section \ref{subsecronaldo} for further details.) 

{}

More generally, suppose that a symplectic manifold $M$ carries an effective Hamiltonian action of a torus $T^p$ with moment map $\mu_p= (\mu^1,\dots, \mu^p)$. Let 
$\mathcal{P}_F$ be a Fourier polarization on $M$, so that its real directions are given by the orbits of $T^p$, that is they are generated by the Hamiltonian vector fields of the components of $\mu_p$, and the holomorphic directions give an ample K\"ahler structure on the coisotropic reductions of $\mathcal{P}_F.$ For this type of mixed polarizations,  with respect to the $T^p$-action, the quantization commutes with reduction correspondence becomes, almost, a tautology.

Indeed, given what the real and holomorphic directions in $\mathcal{P}_F$ are, one expects that, generally, $\mathcal{P}_F$-polarized sections should have the form of sums of products of a factor given by a function of the components of $\mu_p$ times a factor holomorphic along the coisotropic reductions. 

Reducing after quantizing just corresponds to fixing the values of the components of $\mu_p$ in the polarized sections, according to the chosen level set of $\mu_p$. Moreover, since the components of $\mu_p$ are $\mathcal{P}_F$-polarized, their Kostant-Souriau prequantum operators will act on $\mathcal{P}_F$-polarized sections just by multiplication operators, constant on each level set of $\mu_p$. 
Thus, after reduction they will act just by multplication by the appropriate constant. This a natural property to demand for states that correspond to states in the quantization of the symplectic reduction. This leads to a spectral decomposition of the Hilbert space of $\mathcal{P}_F$-polarized sections, as described in (\ref{spectral}) in the Introduction.

On the other hand, for a given level set $\mu_p=c$, in the first reduce then quantize approach one then just obtains the states which are holomorphic with respect to the K\"ahler structure on the corresponding coisotropic reduction $M_c = \mu_p^{-1}(c)/T^p$. Thus, we expect, at the level of linear spaces,
\begin{equation}\label{toric spectral}
\mathcal{H}_{\mathcal{P}_F} = \bigoplus_{c} \mathcal{H}_{M_c},    
\end{equation}

where $\mathcal{H}_{M_c}$ is the K\"ahler quantizations of $M_c$ and where the sum goes over quantizable coisotropic reductions of $\mathcal{P}_F$. Thus, quantization commutes with reduction is a natural consequence of the general properties of the Fourier polarization $\mathcal P_F$. If we define a (natural) hermitian structure on $\mathcal{H}_{\mathcal{P}_F}$ which makes (\ref{toric spectral}) unitary then, automatically, 
the quantization commutes with reduction correspondence for the polarization $\mathcal{P}_F$ is unitary.
{}
Typically, the quantization commutes with reduction correspondence is considered for a given $T^p$-invariant K\"ahler structure, $\mathcal{P}_0$, on the ambient symplectic manifold and for the induced structures on the corresponding K\"ahler quotients. If for this K\"ahler polarization one can define an associated Fourier polarization $\mathcal{P}_F$, as above, with matching K\"ahler structures along the K\"ahler and coisotropic reductions, respectively, then the question of unitarity in the quantization commutes with reduction for $\mathcal{P}_0$ can be more naturally formulated simply 
has the problem of whether or not the quantizations with respect to $\mathcal{P}_0$ and with respect to $\mathcal{P}_F$ are unitarily equivalent.

In Section \ref{subsecqcr}, in the toric case, we will show that the polarization $\mathcal{P}_\infty$ obtained in Section \ref{subsec_geodraysconv} is a Fourier polarization where the above description holds, and where the quantization commutes with reduction correspondence is unitary, where the hermitean structure in the quantization with respect to $\mathcal{P}_\infty$ is induced in a natural way from the family of toric K\"ahler polarizations $\mathcal{P}_t, t\geq 0$, leading to $\mathcal{P}_\infty$ (see Section \ref{subsecqcr}). In this case, moreover, the unitary correpondence is obtained for all coisotropic reductions at once, with no need for choosing level-dependent constants to achieve unitarity. The fact that for a given initial toric K\"ahler polarization $\mathcal{P}_0$ the quantization commutes with reduction correspondence is not unitary becomes, more naturally, the fact that quantizations with respect to $\mathcal{P}_0$ and $\mathcal{P}_\infty$ are not unitarily equivalent. Note that unitarity in quantization commutes with reduction holds only in the case when $M=(\mathbb{C}^*)^n$ and the toric K\"ahler structure is flat.

{}

\subsection{Some properties of the K\"ahler reductions of $\mathcal{P}_\infty$}
\label{newnewsub}

Of course, the symplectic reductions discussed above in Section \ref{newsubsec} may be singular. Let $V_c\subset {\mathbb R}^n$ denote the hyperplace defined 
by $\mu_p = (x^1,\dots, x^p) = c\in \mathbb{R}^{p}$ and let 
$M_c = \mu_p^{-1}(V_c)/T^p.$   
Let us describe, in more detail, the singularities one obtains along the symplectic reductions above.
From \cite{CDG} one has that the symplectic potentials for the reduced K\"ahler toric structures are given by the restriction of the symplectic potential on $P$ to the corresponding level sets of the moment map.
Let then the Delzant polytope $P\subset \mathbb R^n$ be defined by
$$
l_j(x) = \langle x, \nu_j\rangle +\lambda_j \geq 0, \, j=1, \dots, r,
$$
where $\nu_j$ is the primitive inward pointing normal to the facet $j$ of $P$. The associated Guillemin sympletic potential is
$$
g_P = \frac12 \sum_{j=1}^r l_j \log l_j.
$$
Let us consider the symplectic reductions corresponding to the level sets
$$
(x_1, \dots, x_k )=(c_1, \dots, c_k) =c\in \mathbb R^k, \,\, k< n,
$$
which we assume have non-empty intersection with $P$.
Let $\nu_j = (a_j, b_j), \, a_j\in \mathbb Z^k, b_j\in \mathbb Z^{n-k}$ and $y=(x_{k+1}, \dots, x_n)$.
The restriction of the Guillemin potential for $P$ becomes
$$
g_\mathrm{red} = \frac12
\sum_{j=1}^r \tilde l_j \log \tilde l_j,
$$
where 
$$
\tilde l_j (y) = \langle y, b_j\rangle + \tilde \lambda_j,
$$
where $\tilde \lambda_j = \langle c, a_j\rangle+\lambda_j$. 
The reduced polytope $P_\mathrm{red}^c$ is defined 
by the conditions
$$
\tilde l_j (y) \geq 0, j=1, \dots r.
$$
Since the $b_j$ will in general not be primitive, we obtain immediately that, in general,  the level $c$ reduction will have, at least, orbifold singularities. 
By construction, every term in $g_\mathrm{red}$ will be singular somewhere on the boundary of $P^c_\mathrm{red}$ for some value of $c$, since the level sets have non-empty intersection with $P$. That is, for each $j=1, \dots r$ there is a value of $c$ and $y\in P^c_\mathrm{red}$ such that $\tilde l_j(y)=0.$ Since $P$ is Delzant, at most $n$ linear funtions $\tilde l_j$ are allowed to vanish simultaneously at a given $y\in P^c_\mathrm{red}$ but, in general, more than $n-k$ such terms may vanish so that $P^c_\mathrm{red}$ will not, in general, be Delzant. Even when $P^c_\mathrm{red}$ is Delzant, $g_\mathrm{red}$ is not in general of the form
$$
g_{P^c_\mathrm{red}} + \mathrm{smooth}
$$
so that the geometry induced by $g_\mathrm{red}$ will in general be singular and will include singularities not of orbifold type. As a simple example, we can consider the reductions of the form 
$$
x_3=\alpha_1 x_1 + \alpha_2 x_2 + c
$$
in $\mathbb C^3$, giving the reduced symplectic potential, at $c=0$,
$$
g_\mathrm{red}=\frac12 x_1 \log x_1 + \frac12 x_2 \log x_2 + \frac12 (\alpha_1 x_1 + \alpha_2 x_2) \log (\alpha_1 x_1 + \alpha_2 x_2).
$$
which is manifestly singular at the origin $x_1=x_2=0$. In fact, for $\alpha_1=\alpha_2=\alpha$, one finds from Abreu's formula \cite{Ab2} that the scalar curvature reads
$$
S = \frac{2\alpha}{(\alpha+1)(x_1+x_2)}, 
$$
which explodes at the origin.
{}

%In the next sections we will study the quantization along the mixed toric polarizations described above. 
In the event that the hyperplane $V_c$ intersects the moment polytope at integral points, it is natural to define the quantization $\mathcal{H}_{M_c}$ of the Kähler reduction $M_c$ as being isomorphic to the space of distributional sections supported on the integral points of each level set, as in \eqref{eq:section-limit-cst}. If there are no integral points at the intersection, the quantization is then the trivial space $\{0\}$. 
As remarked above, some of these reductions will have the structure of a toric orbifold, while some may have worse singularities, depending on how the hyperplane $V_c$ meets the facets of $P$. It is natural to define the quantization of the level sets containing Bohr-Sommerfeld points as being given by the limit of the K\"ahler polarization, as in Theorem \ref{convsections}, with other level sets having trivial quantization. 
We will thus obtain for the quantization in the limit mixed toric polarization,
\begin{align}\nonumber
  \mathcal{H}_{\mathcal{P}_\infty} \cong \bigoplus_{c \in \Delta\cap {\mathbb Z}^p} \mathcal{H}_{M_c},
\end{align}
so that the quantum Hilbert space $\mathcal{H}_{\mathcal{P}_\infty}$ of the quantization of $M$ in the limit polarization decomposes as the direct sum of the quantization spaces of 
the symplectic reductions. In fact, this isomorphism is unitary for a natural hermitian 
structure on $\mathcal{H}_{\mathcal{P}_\infty}$, as described in section \ref{subsecqcr} and Theorem \ref{thmqcr}.

{}

\subsection{Asymptotic unitarity}
\label{subsecqcr}

 Let us consider the half-form corrected holomorphic section with respect to the complex structure $I_s$, $\sigma^m_s$ 
(Note that due to the integration along $T^n$, $\sigma^m_s$ is orthogonal to $\sigma^{m'}_s$ for $m\neq m'$.) 
From equation (4.7) in \cite{KMN1}, the norm squared of the half-form corrected section $\sigma^m_s$ is
$$
\vert\vert \sigma^m_s\vert\vert^2_{L^2} = 
 \int_P  e^{-2s(\sum_{j=1}^p(x_j-m_j)^2-H)} e^{-2(\sum_{j=1}^n (x_j-m_j)y^j-g)}\cdot
 $$
 $$
 \cdot (\det G_s)^\frac12 dx_1\cdots dx_p\cdot  dx_{p+1}\cdots dx_n,
 $$
 where $y= \partial g / \partial x.$
%We then gave the analog of Lemma 4.12 of %\cite{KMN1}
Let $U_{s}:\mathcal{H}_{\mathcal{P}_{0}} \to \mathcal{H}_{\mathcal{P}_{s}}$ be the generalized coherent state transform (gCST) is defined as before (see equation (\ref{gcst})). We denote $U_{s}(\sigma^{m}_{0})$ by $\tilde{\sigma}^{m}_{s}$. Then we have
 \begin{thm}\label{lemma-norms}
Let $\{\tilde{\sigma}^{m}_{s}:= U_{s}(\sigma_{0}^{m})\}_{m \in P \cap \mathbb{Z}^{n}}$ be the basis of $\mathcal{H}_{\mathcal{P}_{s}}$. Then we have:
$$\lim_{s \to \infty}\vert\vert \tilde{\sigma}^m_s\vert\vert_{L^2}^2 =  c_m \pi^{p/2},$$
where the constant $c_m$ is given by
\begin{align}\label{coefficient}
c_m =  \int_{P}  \left(\Pi_{j=1}^p\delta(x_j-m_j)\right) e^{-2((x-m)\cdot y -g)} (\det D)^\frac12dx_1\cdots dx_n.\hspace{0.5cm}
\end{align}
\end{thm}

\begin{proof}
   By the proof of Lemma 4.12 in \cite{KMN1}, we have, as $s\to \infty$,
$$
\det G_s \sim s^{p} \det D,
$$
where $G_{s}=\mathrm{Hess}g_{s}= \left[ 
    \begin{array}{cc}
        A_1 + sI_p & A_2 \\
        A_3 & D
    \end{array}
    \right]$.
   By equation (4.7) in \cite{KMN1}, the norm squared of the half-form corrected section $\sigma^m_s$ is
$$
\vert\vert \sigma^m_s\vert\vert^2_{L^2} = 
 \int_P  e^{-2s(\sum_{j=1}^p(x_j-m_j)^2-H)} e^{-2(\sum_{j=1}^n (x_j-m_j)y^j-g)} (\det G_s)^\frac12 dx_1\cdots dx_n,
 $$
 where $y= \partial g/ \partial x.$
As $s\to \infty$, the Laplace approximation (see \cite[Lemma4.1]{KMN1}) for the integration on the 
variables $x^1, \dots, x^p,$
gives, asymptotically
$$
\vert\vert \sigma^m_s\vert\vert_{L^2}^2 \sim \left(\frac{2\pi}{s}\right)^{p/2} 
\frac{e^{2sH(m)} s^{p/2} }{2^{p/2}} c_m = \pi^{p/2} c_m e^{2sH(m)}.
{}
$$
By Theorem \ref{local-description}, we have
$$\vert\vert \tilde{\sigma}^m_s\vert\vert^2_{L^2}=\vert\vert \sigma^m_s\vert\vert^2_{L^2}e^{-2sH(m)}.$$
Therefore, we conclude:
$$\lim_{s \to \infty}\vert\vert \tilde{\sigma}^m_s\vert\vert_{L^2}^2 =  c_m \pi^{p/2}.$$
\end{proof}

In view of Theorem \ref{local-description}, Theorem \ref{convsections} and Theorem \ref{lemma-norms}, it is then natural to consider the following
\begin{dfn}\label{def-hermitian}
We define a hermitian structure $\vert\vert\cdot\vert\vert_{(\infty)}$ on quantum space 
$\mathcal{H}_{\mathcal{P}_\infty}=\mathrm{span}\{ \tilde{\sigma}_{\infty}^{m} \}_{m\in P \cap \mathbb{Z}^{n}}$
such that the basis $\{ \tilde{\sigma}_{\infty}^{m} \}_{m\in P \cap \mathbb{Z}^{n}}$ is orthogonal and
$$\vert\vert \tilde{\sigma}_{\infty}^{m} \vert\vert_{(\infty)}^{2}=c_m \pi^{p/2},$$
where the constant $c_m$ is given by (\ref{coefficient}).
\end{dfn}

Recall that, as mentioned in section \ref{newsubsec}, from \cite{CDG}, the restriction of the symplectic potential to a level set of the moment map defines a toric symplectic potential in the corresponding  symplectic reduction. Thus, for a given K\"ahler reduction 
$M_{\underline m}= \mu_p^{-1}(m_1,\dots, m_p)/T^p$, where $\underline m =(m_1, \dots , m_p)$, we will have a natural toric K\"ahler structure defined by taking as symplectic potential the restriction of $g$ to the level set $\mu^{-1}(m_1, \dots, m_p)$. For such a symplectic reduction $M_{\underline m}$, equipped with this K\"ahler structure, consider the corresponding half-form corrected K\"ahler quantization giving the collection of monomial sections 
$\left\{\sigma^{(\underline m, m')}\right\}_{(\underline m,m')\in P\cap \mathbb{Z}^n}$ and the corresponding Hilbert space 
$\mathcal{H}_{M_{\underline m}}$,
given by the standard half-form corrected inner product, as described for symplectic toric manifolds in section \ref{section-prelim}. Note that even in the case when $M_{\underline m}$ is singular, we are considering only the monomial sections corresponding to integral points in that level set and the (singular) toric reduced K\"ahler structure which is smooth on the open dense orbit in $M_{\underline m}$.

As a consequence of Theorem \ref{lemma-norms} and the natural Definition \ref{def-hermitian}, we then obtain that the quantization commutes with reduction correspondence between quantization in the limit mixed polarization and the K\"ahler reductions given by integrable values of $\mu_p$ is, in fact, unitary. As described in Section \ref{newsubsec}, this is to be naturally expected given the properties of $\mathcal{P}_\infty$.

{}

\begin{thm}\label{thmqcr}
    The natural $T^n$-equivariant  linear isomorphism
\begin{eqnarray}\nonumber
\mathcal{H}_{\mathcal{P}_\infty} &\to& \bigoplus_{\underline m\in \Delta\cap\mathbb{Z}^n}\mathcal{H}_{M_{\underline m}}\\ \nonumber
\tilde \sigma^m_\infty & \mapsto & \sigma^{\underline m, m'},
\end{eqnarray}
for $(m_1, \dots, m_p)=\underline m$ and $m=(\underline m, m')\in P\cap \mathbb{Z}^n$
    is unitary up to the overall constant $\pi^{p/2}$.
\end{thm}

\begin{proof}
Let $m=(\underline m, m')\in P\cap \mathbb{Z}^n$.
The norm of $\tilde \sigma^m_\infty$ is given in Definition \ref{def-hermitian}. In the expression (\ref{coefficient}) for $c_m$ the integral over $dx^1\cdots dx^p$ localizes the integrand on 
$x_j=m_j, j=1, \dots, p$. In the exponent, then, only partial derivatives of $g_P$ with respect to 
$(x_{p+1}, \dots, x_{n})$ survive, corresponding to the term $(x-m)\cdot y$. Likewise, the matrix $D$ gives the matrix of second derivatives of $g_P$ with respect to $(x_{p+1}, \dots, x_n)$. Thus,
after integrating along $dx_1\cdots dx_p$ one obtains precisely the expression for the half-form corrected $L^2$ norm of $\sigma^{(\underline m, m')}$ on the reduction $M_{\underline m}$ for the symplectic potential given by the restriction of $g$ to $\mu_p^{-1}(\underline m)$, as in (\ref{reducedsections}).
(Compare with the general expression for the $L^2$ norm of the monomial sections just before Theorem \ref{lemma-norms}).
\end{proof}

Along the Mabuchi geodesic ray defined by the symplectic potentials 
$g_s = g +sH, s>0$, for finite $s$, quantization commutes with reduction is not unitary, a fact that can be easily checked. However, in the infinite geodesic limit $s\to \infty$ one obtains a natural hermitian structure on the Hilbert space of distributional polarized sections with respect to 
$\mathcal{P}_\infty$ which gives, as just described, a natural $T^n$-equivariant unitary isomophism
$$
\mathcal{H}_{\mathcal{P}_\infty} \cong  \bigoplus_{\underline m \in \Delta\cap \mathbb{Z}^p} \mathcal{H}_{M_{\underline m}}.
$$

Note, moreover, that this unitary correspondence holds for all quantizable levels of $\mu_p$ simultaneously, without the need of adjusting overall constants at each level in order to achieve unitarity.
We thus obtain, in the spirit of Section \ref{newsubsec},

\begin{cor}\label{cor_qrunitary}
The quantization commutes with reduction correspondence for the toric mixed polarization $\mathcal{P}_\infty$ is unitary (up to an overall constant). \end{cor}

The property that the quantization commutes with reduction correspondence, relative to the $T^p$ action, is not unitary for the initial K\"ahler polarization $\mathcal{P}_0$ can then be (re)formulated by expressing the fact  that the quantizations with respect to $\mathcal{P}_0$ and to $\mathcal{P}_\infty$ are not unitarily equivalent.

\begin{rmk}
    Note that in the case when $M_{\underline m}$ is singular, as described in section \ref{newnewsub}, we are considering that the K\"ahler quantization is obtained by the monomial sections corresponding to the integral points of $P$ in that $\mu_p$ level set.
\end{rmk}

\begin{rmk}\label{remark42}
    Note that in this paper, for simplicity, we have described the family of toric K\"ahler structures along the Mabuchi geodesic generated by the Hamiltonian $H$, which is quadratic in $(x_1, \dots, x_p).$ In fact, the results generalize straightforwardly to Mabuchi geodesics generated by Hamiltonians which are strictly convex in $(x_1,\dots, x_p)$ but not necessarily quadratic like $H$. (See eg \cite{BFMN} and the proof of Lemma 4.12 in \cite{KMN1}.)
For Mabuchi geodesics starting at the same symplectic potential, both the limit polarization 
$\mathcal{P}_\infty$ and the hermitian structure defined above is then natural in the sense that $\mathcal{P}_\infty$ is also obtained in the limit of infinite geodesic time for this larger family of geodesics and that the result of in Theorem \ref{lemma-norms}  also holds so, that the hermitian structure naturally induced on  $\mathcal{H}_{\mathcal{P}_\infty}$ is the same as above. Note that the corresponding gCST (see \cite{KMN4}) will then satisfy
$$
\lim_{s\to \infty} \vert\vert U_s \left(\sigma^m_0\right)\vert\vert_{L^{2}}^{2} =  \vert\vert\tilde{\sigma}^m_\infty\vert\vert_{(\infty)}.
$$
\end{rmk}

\subsection{Example: $M= \mathbb{CP}_1\times \mathbb{CP}_1$}\label{p1p1example}

Let $M=\mathbb{CP}_1\times \mathbb{CP}_1$ be the toric manifold associated to the Delzant polytope  
\[
P = \left\{ (x_1, x_2) :\ -\frac{3}{2} < x_1 < \frac{3}{2},\quad -\frac{7}{2} < x_2 < \frac{1}{2} \right\}.
\]
Then the symplectic potential is given by:
\begin{align*}
g(x_1, x_2) = &\frac{1}{2} \Big[
(x_1 + \tfrac{3}{2}) \log(x_1 + \tfrac{3}{2})
+ (\tfrac{3}{2} - x_1) \log(\tfrac{3}{2} - x_1)\\
&+ (x_2 + \tfrac{7}{2}) \log(x_2 + \tfrac{7}{2})
+ (\tfrac{1}{2} - x_2) \log(\tfrac{1}{2} - x_2)
\Big].
\end{align*}
The associated toric Kähler structure is described in action-angle coordinates \((x_1, x_2, \theta_1, \theta_2)\) on \( \mathring{P} \times T^2 \). The compatible complex structure \(J\) is given by

\[
J =
\begin{bmatrix}
    0 & -G^{-1} \\
    G & 0
    \end{bmatrix}\]
where $G=\mathrm{Hess}g=\begin{bmatrix}
    a & 0 \\
    0 & b
    \end{bmatrix}$
with
\[
a = \frac{1}{2}\left( \frac{1}{x_1 + \frac{3}{2}} + \frac{1}{\frac{3}{2} - x_1} \right),\qquad
b = \frac{1}{2}\left( \frac{1}{x_2 + \frac{7}{2}} + \frac{1}{\frac{1}{2} - x_2} \right).
\]
Let $H$ be a smooth function that is strictly  convex along $x_1 \in [-\frac{3}{2}, \frac{3}{2}]$. For simplicity, let us take $H=\frac{1}{2}x_{1}^{2}$.
Let $\shP_{s}$ be the family of K\"ahler polarization associated to the symplectic potential $g_{s}=g+sH$. Then over the open dense $(\CC^{*})^{2}$, the $J_{s}$-holomorphic coordinates are given by
\[
z_s^1 = \frac{\partial g_s}{\partial x_1} + i\theta_1 = z_1 + s x_1 \quad \text{and} \quad z_s^2 = \frac{\partial g_s}{\partial x_2} + i\theta_2 = z_2.
\]
Then 
\(
\mathcal{P}_s = \left\langle X_{\bar{z}_s^1},\, X_{\bar{z}_s^2} \right\rangle_{\mathbb{C}}.
\)
By direct computation (see Theorem \ref{prop:evol-polarization}), we have: $$\lim_{s\rightarrow \infty}\shP_{s}=\lim_{s\rightarrow \infty} \langle X_{\bar{z}_{s}^{1}},X_{\bar{z}_{s}^{2}} \rangle_{\mathbb C}=\lim_{s\rightarrow \infty} \langle X_{x_{1}}+\frac{1}{s}X_{\bar{z}_{1}},X_{\bar{z}_{2}} \rangle_{\mathbb C}=\langle \frac{\partial}{\partial \theta_{1}},X_{\bar{z}_{2}} \rangle_{\mathbb C}=\mathcal{P}_\infty.$$
Let \(\{\sigma_{0}^{m}\}_{m\in P\cap\mathbb{Z}^{2}}\) denote the canonical basis of \(J\) (i.e., \(J_0\))-holomorphic sections.
Following \cite{KMN1}, consider the equivariant prequantum line bundle \(L\) determined by the Delzant polytope \(P\).
The half-form corrected Hilbert space \(\mathcal{H}_{\mathcal{P}_s}\), associated to the polarization \(\mathcal{P}_s\), is given by
\begin{align*}\label{kahlerquants}
    \mathcal{H}_{\mathcal{P}_s} = \mathrm{Span}_{\mathbb{C}}\left\{
        \sigma_s^m = w_s^m \, e^{-\kappa_s(x)} \mathbf{1}^{U(1)} \otimes \sqrt{dZ_s} :\  m \in P \cap \mathbb{Z}^2
    \right\},
\end{align*}
where the Kähler potential is
\[
\kappa_s = x \cdot \frac{\partial g_s}{\partial x} - g_s = \kappa + \frac{1}{2}s x_1^2,
\]
with \(\kappa = x \cdot \frac{\partial g}{\partial x} - g\). Here, \(w_s^j = e^{z_s^j}\), \(dZ_s = dz_s^1 \wedge dz_s^2\), and \(\mathbf{1}^{U(1)} \otimes \sqrt{dZ_s}\) is a trivialization of \(L\) over the open dense subset \((\mathbb{C}^*)^2\).

In particular, for \(s=0\) and \(m = (m_1, m_2) \in P \cap \mathbb{Z}^2\), letting \(w_j = e^{z_j}\) for \(j=1,2\), we have
\[
\sigma_0^m = w_1^{m_1} w_2^{m_2} e^{-\kappa(x)} \mathbf{1}^{U(1)} \otimes \sqrt{dz_1 \wedge dz_2}.
\]
Let $U_{s}= (e^{s\hat{H}^{prQ}} \otimes e^{is \shL_{X_H}}) \circ e^{-s\hat H^Q}$ be the generalized coherent state transform (gCST) as discussed in section \ref{sec-clb}.
By Theorem \ref{local-description}, we have: $$U_{s}\sigma_{0}^{m}=e^{-sH(m)}\sigma_{s}^{m}=e^{-\frac{1}{2}sm_{1}^{2}}\sigma_{s}^{m}.$$
Then by Theorem \ref{convsections}, for $m=(m_{1},m_{2})\in P\cap \mathbb{Z}^2$, we have:

\begin{align*}
&\tilde{\sigma}^m_\infty=
 \lim_{s\to +\infty} U_{s}\sigma^m_{0} =\lim_{s\to +\infty} e^{-\frac{1}{2}sm_{1}^{2}}\sigma_{s}^{m}\\
 &=\lim_{s\to +\infty} e^{-\frac{1}{2}sm_{1}^{2}}w_s^me^{-(\kappa+\frac{1}{2}sx_{1}^{2})} \otimes \sqrt{d(sx_{1}+z_{1})\wedge dz_{2}}\\
 &=\lim_{s\to +\infty}  \left(\sqrt{\frac{s}{2\pi}}e^{- \frac{s}{2}(x_1-m_1)^2}\right)\sqrt{2\pi}e^{-\kappa}w_{2}^{m_{2}} e^{m_{1}\frac{\partial g}{\partial x_{1}}}e^{im_{1}\theta_{1}}\sqrt{ dx_{1}\wedge dz_{2} +O(\frac{1}{s})}\\
 &=\sqrt{2\pi}\delta^{m_{1}}e^{-\kappa}w_{2}^{m_{2}} e^{m_{1}\frac{\partial g}{\partial x_{1}}}e^{im_{1}\theta_{1}}\sqrt{ dx_{1}\wedge dz_{2}}.
\end{align*}
Let \(\{\tilde{\sigma}^{m}_{s} := U_{s}(\sigma_{0}^{m})\}_{m \in P \cap \mathbb{Z}^{n}}\) denote a basis for \(\mathcal{H}_{\mathcal{P}_{s}}\). 
According to equation (4.7) in \cite{KMN1}, the squared \(L^2\)-norm of the section \(\tilde{\sigma}^m_s\) is given by
\[
\left\| \tilde{\sigma}^m_s \right\|^2_{L^2} = 
 \int_P  e^{-s(x_{1}-m_1)^2} e^{-2(x-m)\cdot y - g}
 \cdot (\det G_s)^{1/2} \, dx_1 dx_2,
\] 
 where $y= \frac{\partial g}{\partial x}$ and $G_{s}=\mathrm{Hess}(g_{s})= \left[ 
    \begin{array}{cc}
        a + s & 0 \\
        0 & b
    \end{array}
    \right].$
By Theorem \ref{lemma-norms} we have:
$$\lim_{s \to \infty}\vert\vert \tilde{\sigma}^m_s\vert\vert_{L^2}^2 =  c_m \pi^{p/2},$$
where the constant $c_m$ is given by
\begin{align*}
c_m =  \int_{P} \delta(x_{1}-m_{1}) e^{-2((x_{2}-m_{2})\cdot \frac{\partial g}{\partial x_{2}} -g)} b^\frac12dx_{1}dx_{2}.\hspace{0.5cm}
\end{align*}
By equation (\ref{reducedsections}),
we know that  
\begin{align*}
\sigma^{\underline m_{1}, m_{2}} = e^{-\kappa_\mathrm{red}} w_{2}^{m_{2}} \sqrt{dz_{2}}=e^{-(x_{2}\cdot \frac{\partial g}{\partial x_{2}} -g)} w_{2}^{m_{2}} \sqrt{dz_{2}}.
\end{align*}
By direct computation, we have:
$$\vert\vert \sigma^{\underline m_{1}, m_{2}}\vert\vert_{L^2}^2=\int_{-\frac{7}{2}}^{\frac{1}{2}}e^{-2((x_{2}-m_{2})\cdot \frac{\partial g}{\partial x_{2}} -g)} b^\frac12dx_{2}=c_{m}.$$
Note that $\vert\vert \tilde{\sigma}_{\infty}^{m} \vert\vert_{(\infty)}^{2}=c_m \pi^{p/2}$  (see Definition \ref{def-hermitian}). Therefore, the natural $T^{n}$-equivariant linear isomorphism mapping $\tilde{\sigma}_{\infty}^{m} $ to $\sigma^{\underline m_{1}, m_{2}}$ is unitary up to the overall constant factor $\pi^{1/2}$ (see Theorem \ref{thmqcr}).

%\begin{center}
\section{Acknowledgements}The authors thank the referee for insightful comments and suggestions for the organization of the material in the paper which have significantly improved its clarity. 
The authors would also like to thank Thomas Baier for many discussions on the topics of the paper. The authors were supported by the the projects CAMGSD UIDB/04459/2020 and CAMGSD  UIDP/04459/2020. A.P. held an FCT LisMath PhD fellowship PD/BD/135528/2018.   
%\end{center}

\section{Declarations}
\textbf{Conflict of interest:} We hereby declare that there are no conflicts of interest associated with our paper entitled ``A new look at unitarity in quantization commutes with reduction for toric manifolds".

\textbf{Data Availability:} This paper has no associated data.

%\subsection{Representation-theoretic meaning of mixed polarizations}
%\clearpage

\end{document}